\documentclass[reqno,12pt]{amsart}
\usepackage{amsmath,amssymb,latexsym,soul,cite,mathrsfs}
\usepackage{blindtext}
\usepackage{enumitem}
\usepackage[utf8]{inputenc}
\usepackage{amsthm,amsfonts}
\usepackage{comment}
\usepackage[mathcal]{eucal}
\usepackage{latexsym}
\usepackage[utf8]{inputenc}
\usepackage{bm}
\usepackage[all]{xy}
\usepackage{xcolor}
\usepackage{graphicx}
\usepackage{hyperref}

\usepackage[hyperpageref]{backref}

\newtheorem{theorem}{Theorem}

\newtheorem{theo}{Theorem}[section]

\newtheorem{lemma}{Lemma}[section]

\newtheorem{prop}{Proposition}[section]

\newtheorem{cor}{Corollary}[section]

\numberwithin{equation}{section}

\title[EXACT GROWTH ON WEIGHTED SOBOLEV SPACES]{SHARP HIGHER ORDER ADAMS' INEQUALITY WITH EXACT GROWTH CONDITION ON WEIGHTED SOBOLEV SPACES}

\author[J.M.\ do \'O]{Jo\~ao Marcos do \'O}
\author[G. Lu]{Guozhen Lu}
\author[R.~Ponciano]{Raon\'{\i} Ponciano}

\address[Jo\~{a}o Marcos do \'O]{Dep. Mathematics,
	Federal University of Para\'{\i}ba
	\newline\indent
	58051-900, Jo\~ao Pessoa-PB, Brazil}
\email{\href{mailto:jmbo@mat.ufpb.br}{jmbo@mat.ufpb.br}}

\address[Guozhen Lu]{Dep. Mathematics, University of Connecticut
	\newline\indent
	06269, Storrs-CT, United States of America}
\email{\href{mailto:guozhen.lu@uconn.edu}{guozhen.lu@uconn.edu}}

\address[Raon\'{\i}~Ponciano]{Dep. Mathematics,
	Federal University of Para\'{\i}ba
	\newline\indent
	58051-900, Jo\~ao Pessoa-PB, Brazil}
\email{\href{mailto:raoni.cabral.ponciano@academico.ufpb.br}{raoni.cabral.ponciano@academico.ufpb.br}}

\thanks{This work was partially supported by Conselho Nacional de Desenvolvimento Cient\'ifico e Tecnol\'ogico (CNPq) \#312340/2021-4 and \#429285/2016-7, Funda\c c\~ao de Apoio \`a Pesquisa do Estado da Para\'iba (FAPESQ) \#2020/07566-3, Coordena\c c\~ao de Aperfei\c coamento de Pessoal de N\'ivel Superior (CAPES) \#88887.633572/2021-00 and Simons collaboration grants 519099 and 957892 and a Simons Fellowship grant 1030658 from the Simons foundation}



\begin{document}

\begin{abstract}
This paper introduces a novel higher order Adams inequality that incorporates an exact growth condition for a class of weighted Sobolev spaces. Our rigorous proof confirms the validity of this inequality and provides insights into the optimal nature of the critical constant and the exponent within the denominator. Furthermore, we apply this inequality to study a class of ordinary differential equations (ODEs), where we successfully derive both a concept of the weak solution and a comprehensive regularity theory.
\end{abstract}

\maketitle

\section{Introduction}

Our main goal is to establish an Adams-type inequality for a class of weighted Sobolev spaces, including the exact growth condition and sharp constants. Before we present our main results, let us introduce some related classical results.  For the case of the first derivative, J. Moser \cite{MR0301504} refined a result by N. Trudinger \cite{MR0216286} (see also \cite{MR0140822} and \cite{MR0192184}) proving the following theorem.
\begin{theorem}\label{theomoser}
Let $\Omega\subset\mathbb R^N$ with $N\geq2$ and $u\in W^{1,N}(\Omega)$ such that $\|\nabla u\|_{L^{N}(\Omega)}\leq1$. Then there exists $C(N)>0$ such that
\begin{equation*}
\int_\Omega\exp(\alpha|u|^{\frac{N}{N-1}})\mathrm dx\leq C(N)|\Omega|
\end{equation*}
for any $\alpha\leq\alpha_N:=N\omega_{N-1}^{N-1}$, where $\omega_{N-1}$ is the area of the surface of the unit ball in $\mathbb R^N$. The constant $\alpha_N$  is sharp because if $\alpha>\alpha_N$, then the integral on the left can be arbitrarily large by an appropriate choice of $u$.
\end{theorem}

Several generalizations of Theorem \ref{theomoser} have been developed; see, for instance, \cite{MR1163431,MR3053467,MR1865413,MR0960950} and references therein. When $\Omega=\mathbb R^N$, many results have been aimed to obtain a similar estimate as in Theorem \ref{theomoser}.
For such cases, we need to replace the function $\exp$ by its series representative without the first terms
\begin{equation*}
\exp_N(t):=\mathrm e^t-\sum_{j=0}^{N-2}\dfrac{t^j}{j!},\quad N\in\mathbb N,
\end{equation*}
since we only have the embedding $W^{1,N}(\mathbb R^N)\hookrightarrow L^q(\mathbb R^N)$ for $N\leq q<\infty$. Consequently, there exists $u\in W^{1,N}(\mathbb R^N)$ such that $\int_{\mathbb R^N}\exp(\alpha|u|^{\frac{N}{N-1}})\mathrm dx=\infty$ for any $\alpha>0$.
We also cite the work due to  S. Adachi and K. Tanaka \cite{MR1646323} (see also \cite{MR1163431,MR1704875}) for a class of  scale invariant inequality in whole  $\mathbb R^N$ involving the Dirichlet norm $\|\nabla u\|_{L^N(\mathbb R^N)},$
\begin{theorem}\label{theoat}
Let $0<\alpha<\alpha_N$. There exists a constant $C(\alpha,N)>0$ such that
\begin{equation*}
\sup_{\underset{\|\nabla u\|_{L^N(\mathbb R^N)}\leq1}{u\in W^{1,N}(\mathbb R^N)}}\int_{\mathbb R^N}\exp_N(\alpha|u|^{\frac{N}{N-1}})\mathrm dx\leq C(\alpha,N)\|u\|_{L^N(\mathbb R^N)}^N.
\end{equation*}
Moreover, the constant $\alpha_N$ is sharp in the sense that if $\alpha\geq\alpha_N$, the supremum is infinite.
\end{theorem}
In \cite{MR2400264}, Y. Li and B. Ruf  obtained an inequality involving the full Sobolev norm $\|u\|_{W^{1,N}(\mathbb R^N)}:=\|\nabla u\|_{L^N(\mathbb R^N)}+\|u\|_{L^N(\mathbb R^N)}$:
\begin{theorem}\label{theolr}
There exists a constant $C(N)>0$ such that for any domain $\Omega\subset\mathbb R^N$,
\begin{equation*}
\sup_{\underset{\|u\|_{W^{1,N}(\Omega)}\leq1}{u\in W^{1,N}(\Omega)}}\int_\Omega\exp_N(\alpha_N|u|^{\frac{N}{N-1}})\mathrm dx\leq C(N).
\end{equation*}
Moreover, the constant $\alpha_N$ is sharp in the sense that if $\alpha_N$ is replaced by any $\alpha>\alpha_N$, the supremum is infinite.
\end{theorem}

We mention that a symmetrization-free argument without using the P\'olya-Szeg\"{o} inequality was found for both critical and subcritical Moser-Trudinger inequalities by N. Lam et al. in \cite{MR3053467}, \cite{MR2980499} and \cite{MR3130507}. More recently, in \cite{MR3269875,MR3729597,MR4117991} the authors were able to prove that results of the type of Theorem \ref{theoat}, \ref{theolr} and the inequaliies proved in \cite{MR1163431,MR1704875}  are indeed equivalent.

Remember that the Trudinger-Moser problem arose from the search for the optimal Orlicz space $L_\Phi$ such that $W^{1,N}_0(\Omega)\hookrightarrow L_\Phi(\Omega)$, since $W^{1,N}_0(\Omega)\hookrightarrow L^q(\Omega)$ for all $q\geq1$ does not provide the best Orlicz space. Also, note that Theorem \ref{theoat} does not achieve the best constant $\alpha_N$, and Theorem \ref{theolr} achieves it only when assuming the full norm $\|u\|_{W^{1,N}(\mathbb R^N)}\leq1$. Therefore, a natural question arises:
\begin{center}
\textit{Is it possible to achieve the best constant $\alpha_N$ when\\we only require the restriction $\|\nabla u\|_{L^N(\mathbb R^N)}\leq1$?}
\end{center}
S. Ibrahim, N. Masmoudi, and K. Nakanishi \cite{MR3336837} gave a positive answer to this question for the two dimensional case. Afterward, N. Masmoudi and F. Sani \cite{MR3355498} generalized the result in $\mathbb{R}^N$ for all dimensions $N$ and  H. Tang and the second author \cite{MR3472818} established independently
the same result in hyperbolic space $\mathbb{H}^N$ for all dimension $N$. We only state it in the Euclidean space $\mathbb{R}^N$ here.

\begin{theorem}\label{theod}
There exists a constant $C(N)>0$ such that for any $u\in W^{1,N}(\mathbb R^N)$ with $\|\nabla u\|_{L^N(\mathbb R^N)}\leq1$,
\begin{equation*}
\int_{\mathbb R^N}\dfrac{\exp_N(\alpha_N |u|^{\frac{N}{N-1}})}{(1+|u|)^{\frac{N}{N-1}}}\mathrm dx\leq C(N)\|u\|_{L^N(\mathbb R^N)}^N.
\end{equation*}
Moreover, this inequality fails if the power $\frac{N}{N-1}$ in the denominator is replaced with any $p<\frac{N}{N-1}$.
\end{theorem}
 
In the paper \cite{MR3943303}, N. Lam et al established the singular version of Trudinger-Moser inequality with the exact growth under mixed norm in the Soblev type spaces $u\in D^{N,q}(\mathbb{R}^N)$, where $D^{N,q} (\mathbb{R}^{N}),~q\geq1$ is the completion of $C_{0}^{\infty}\left(\mathbb{R}^{N}\right)  $ under the norm $\left\Vert \nabla u\right\Vert _{N}+\left\Vert
u\right\Vert _{q}$. We note that when $q=N$, then
$D^{N,q}\left(\mathbb{R}^{N}\right)  =W^{1,N}\left(\mathbb{R}^{N}\right)$.
 Then the following singular Trudinger-Moser inequality with exact growth was established by N. Lam, L. Zhang, and the second author in \cite[Theorem 15]{MR3943303}.

\begin{theorem}
Let $N\geq2$, $0\leq \beta<N$, and $p\geq q\geq1$. Then there exists a constant
$C=C\left(  N,p,q,\beta\right)  >0$ such that for all $u\in D^{N,q}\left(\mathbb{R}^{N}\right)$ with $\left\Vert \nabla u\right\Vert _{N}\leq1,$ there holds
\begin{equation*}
{\displaystyle\int\limits_{\mathbb{R}^{N}}}\frac{\Phi_{N,q,\beta}\left(  \alpha_{N}\left(  1-\frac{\beta}{N}\right)u^{\frac{N}{N-1}}\right)  }{\left(  1+\left\vert u\right\vert ^{\frac{p}{N-1}\left(  1-\frac{\beta}{N}\right)  }\right)  \left\vert x\right\vert
^{\beta}}dx\leq C\left\Vert u\right\Vert _{q}^{q\left(  1-\frac{\beta}{N}\right)  },
\end{equation*}
\textit{where $\Phi_{N,q,\beta}$ is the Taylor series of the exponential without the first terms given by the Equation (F) in \cite{MR3943303}. Moreover, the inequality does not hold when }$p<q.$
\end{theorem}

In the paper \cite{MR3587065}, affine Trudinger-Moser inequalities with the exact growth were proved by N. Lam and the second author (see Theorem 1.3 in \cite{MR3587065}).

 N. Masmoudi and F. Sani \cite{MR3225631} also derived the following second order Adams' inequality with the exact growth condition in $\mathbb R^4$:
\begin{theorem}\label{theoms}
There exists a constant $C>0$ such that for any $u\in W^{2,2}(\mathbb R^4)$ with $\|\Delta u\|_{L^2(\mathbb R^4)}\leq1$,
\begin{equation*}
    \int_{\mathbb R^4}\dfrac{e^{32\pi^2u^2}-1}{(1+|u|)^2}\mathrm dx\leq C\|u\|^2_{L^2(\mathbb R^2)}.
\end{equation*}
Moreover, this fails if the power 2 in the denominator is replaced with any $p<2$.
\end{theorem}
H. Tang, M. Zhu, and the second author  \cite{MR3405815} established the second order Adams' inequality with the exact growth condition in $\mathbb{R}^N$ for all dimensions $N\geq3$ by demonstrating the following  results:
\begin{theorem}\label{theolu}
There exists a constant $C(N)>0$ such that for all $u\in W^{2,\frac{N}2}(\mathbb R^N)$ ($N\geq3$) satisfying $\|\Delta u\|_{L^{\frac{N}2}(\mathbb R^N)}\leq1$,
\begin{equation*}
\int_{\mathbb R^N}\dfrac{\exp_N(\beta_N|u|^{\frac{N}{N-2}})}{(1+|u|)^{\frac{N}{N-2}}}\leq C(N)\|u\|_{L^{\frac{N}{2}}(\mathbb R^N)}^{\frac{N}2},
\end{equation*}
where $\beta_N=\frac{N}{\omega_{N-1}}\left(\frac{\pi^{\frac{N}2}4}{\Gamma(\frac{N}2-1)}\right)^{\frac{N}{N-2}}$.
\end{theorem}
\begin{theorem}\label{theolu2}
If the power $\frac{N}{N-2}$ in the denominator is replaced by any $p<\frac{N}{N-2}$, there exists a sequence of functions $(u_n)$ such that $\|\Delta u_n\|_{L^{\frac{N}2}(\mathbb R^N)}\leq1$, but
\begin{equation*}
\dfrac{1}{\|u_n\|^{\frac{N}2}_{L^{\frac{N}2}(\mathbb R^N)}}\int_{\mathbb R^N}\dfrac{\exp_N(\beta_N|u_n|^{\frac{N}{N-2}})}{(1+|u_n|)^p}\mathrm dx\to\infty.
\end{equation*}
Moreover, if $\beta>\beta_N$, there exists a sequence of functions $(u_n)$ such that $\|\Delta u_n\|_{L^{\frac{N}2}(\mathbb R^N)}\leq1$, but
\begin{equation*}
\dfrac{1}{\|u_n\|^{\frac{N}2}_{L^{\frac{N}2}(\mathbb R^N)}}\int_{\mathbb R^N}\dfrac{\exp_N(\beta|u_n|^{\frac{N}{N-2}})}{(1+|u_n|)^p}\mathrm dx\to\infty,
\end{equation*}
for any $p\geq0$.
\end{theorem}

A singular version of Adams' inequality with the exact growth was also proved in $W^{2, 2}(\mathbb{R}^4)$ by N. Lam and G. Lu in \cite{MR3587065} (see Theorem 1.2 there).

\begin{theorem}
Let $0\leq\beta<4$ and
$0<\alpha\leq32\pi^{2}\left(  1-\frac{\beta}{4}\right)  $. Then there exists a
constant $C=C\left(  \alpha,\beta\right)  >0$ such that
\[
{\displaystyle\int\limits_{\mathbb{R}^{4}}}\frac{e^{\alpha u^{2}}-1}{\left(  1+\left\vert u\right\vert ^{2-\beta/2}\right)  \left\vert x\right\vert ^{\beta}}dx\leq C\left\Vert u\right\Vert_{2}^{2-\frac{\beta}{2}}~\text{for all }u\in W^{2,2}\left(  \mathbb{R}^{4}\right)  :\left\Vert \Delta u\right\Vert _{2}\leq1.
\]
\textit{Moreover, the power }$2-\beta/2$\textit{ in the denominator cannot be
replaced with any }$q<2-\beta/2$\textit{.}
\end{theorem}

To conclude, we note that
N. Masmoudi and F. Sani \cite{MR3848068} obtained the exact growth inequality with higher order derivatives by considering the $k$-generalized gradient
\begin{equation*}
\nabla^ku=\left\{\begin{array}{ll}
     (-\Delta)^{\frac{k}{2}}u,&\mbox{if }m\mbox{ is even},  \\
     \nabla(-\Delta)^{\frac{k-1}2}u,&\mbox{if }m\mbox{ is odd},
\end{array}\right.
\end{equation*}
and the critical value $\beta_{N,k}$ defined by
\begin{equation*}
\beta_{N,k}=\dfrac{N}{\omega_{N-1}}\left\{\begin{array}{ll}
     \left[\dfrac{\pi^{\frac{N}2}2^k\Gamma\left(\frac{k}2\right)}{\Gamma\left(\frac{N-k}2\right)}\right]^{\frac{N}{N-k}}&\mbox{if }m\mbox{ is even},  \\
     \left[\dfrac{\pi^{\frac{N}2}2^k\Gamma\left(\frac{k+1}2\right)}{\Gamma\left(\frac{N-k+1}2\right)}\right]^{\frac{N}{N-k}}&\mbox{if }m\mbox{ is odd},
\end{array}\right.
\end{equation*}
$\omega_{N-1}$ is the surface measure of the unit $N$-ball. Let $\lceil x\rceil:=\min\{n\in\mathbb Z\colon n\geq x\}$ denote the celling function. They precisely proved the following theorem.

\begin{theorem}\label{theoexacthigher}
Let $k$ be a positive integer with $2<k<N$. There exists a constant $C_{N,k}>0$ such that
\begin{equation*}
\int_{\mathbb R^N}\dfrac{\exp_{\lceil\frac{N}{k}-2\rceil}\left(\beta_{N,k}|u|^{\frac{N}{N-k}}\right)}{(1+|u|)^{\frac{N}{N-k}}}\mathrm dx\leq C_{N,k}\|u\|^{\frac{N}k}_{L ^{\frac{N}k}(\mathbb R^N)},
\end{equation*}
for all $u\in W^{k,\frac{N}{k}}(\mathbb R^N)$ with $\|\nabla^ku\|_{L^{\frac{N}k}(\mathbb R^N)}\leq1$. Moreover, the above inequality fails if the power $\frac{N}{N-k}$ in the denominator is replaced with any $p<\frac{N}{N-k}$.
\end{theorem}

We also mention in passing that Trudinger-Moser type inequality with exact growth with Riesz type potential have also been studied in \cite{Qin} and \cite{MQ}, and the Trudinger-Moser and Adams trace inequalities with exact growth on half spaces were studied in \cite{CLYZ}.

\subsection{Main Results}

In this work, we prove an inequality involving exact growth for the weighted Sobolev space $X^{k,p}_\infty$, which will be defined in the following paragraph. In recent times, the weighted Sobolev space $X^{k,p}_R$ has been extensively studied due to its applicability in radial elliptic problems for the operator $Lu=-r^{-\gamma}(r^\alpha|u'|^\beta u')'$ which includes the $p$-Laplacian and the $k$-Hessian as a particular case. For more details, see \cite{MR1422009,MR3670473,MR1929156,MR1069756,MR1982932}, and references therein. Specifically, by choosing the parameters properly, we have
\begin{flushleft}
    $\mathrm{(i)}$ $L$ is the Laplacian for $\alpha=\gamma=N-1$ and $\beta=0$;\\
    $\mathrm{(ii)}$ $L$ is the $p$-Laplacian for $\alpha=\gamma=N-1$ and $\beta=p-2$;\\
    $\mathrm{(iii)}$ $L$ is the $k$-Hessian for $\alpha=N-k$, $\gamma=N-1$ and $\beta=k-1$.
\end{flushleft}
Therefore, the importance of this space arises from the fact that $X^{1,p}_{0,R}$ is a suitable space to work on problems like:
\begin{equation*}
\left\{\begin{array}{l}
Lu=f(r,u)\mbox{ in }(0,R),  \\
u'(0)=u(R)=0,\\
u>0\mbox{ in }(0,R).
\end{array}\right.
\end{equation*}

Before we state our results, for easy reference, we introduce some notations and the functions spaces used throughout this work.
For each nonnegative integer $\ell$ and $0<R\leq\infty$, let $AC_{\mathrm{loc}}^\ell(0,R)$ be the set of all functions $u\colon(0,R)\to \mathbb R$ such that $u^{(\ell)}\in AC_{\mathrm{loc}}(0,R)$, where $u^{(\ell)}=\mathrm d^\ell u/\mathrm dr^\ell$. For $p\geq1$ and $\alpha$ real numbers, we denote by $L^p_\alpha=L^p_\alpha(0,R)$ the weighted Lebesgue
 space  of measurable functions $u\colon(0,R)\to\mathbb R$ such that
\begin{equation*}
\|u\|_{L^p_\alpha}=\left(\int_0^R|u|^pr^\alpha\mathrm dr\right)^{1/p}<\infty,
\end{equation*}
which is a Banach space under the standard norm $\|u\|_{L_\alpha^p}$.

For any positive integer $k$ and $(\alpha_0,\ldots,\alpha_k)\in \mathbb R^{k+1}$, with $\alpha_j>-1$ for $j=0,1,\ldots,k$, in \cite{MR4112674}, the authors considered the weighted Sobolev spaces for higher order derivatives $X^{k,p}_{0,R}=X^{k,p}_{0,R}(\alpha_0,\ldots,\alpha_k)$ given by all functions $u\in AC^{k-1}_{\mathrm{loc}}(0,R)$ such that
\begin{equation*}
\lim_{r\to R}u^{(j)}(r)=0,\quad j=0,\ldots,k-1\mbox{ and }u^{(j)}\in L^p_{\alpha_j},\quad j=0,\ldots,k.
\end{equation*}
Recently, in \cite{arXiv:2302.02262} was considered the weighted Sobolev spaces for higher order derivatives without boundary conditions denoted by
\begin{equation*}
X_R^{k,p}\!=\!X_R^{k,p}(\alpha_0,\ldots,\alpha_k)\!=\!\{u\colon(0,R)\to\mathbb R : u^{(j)}\in L^p_{\alpha_j},\ j=0,1,\ldots,k\},
\end{equation*}
for any positive integer $k$ and $(\alpha_0,\ldots,\alpha_k)\in\mathbb R^{k+1}$. Using \cite[Proposition 2.2]{arXiv:2302.02262}, one can obtain $u\in AC_{\mathrm{loc}}^{k-1}(0,R)$ for all $u\in X^{k,p}_R$. The spaces $X_R^{k,p}$ and $X^{k,p}_{0,R}$ are complete under the norm
\begin{equation*}
\|u\|_{X_R^{k,p}}=\left(\sum_{j=0}^k\|u^{(j)}\|^p_{L^p_{\alpha_j}}\right)^{1/p}.
\end{equation*}
The study of weighted Sobolev spaces has attracted significant attention due to their essential role in understanding various partial differential equations that involve radial functions. For more details about the theory of weighted Sobolev spaces and embeddings related to such spaces, please refer to \cite{MR2838041, MR3209335, MR3575914, MR3957979, arXiv:2108.04977, arXiv:2203.14181, MR1929156} and references therein.

Considering the elliptic operator $L_{\theta,\gamma}u=-r^{-\theta}(r^{\gamma}u')'$ (which corresponds to the Laplacian when $\theta=\gamma=N-1$ and $u$ is radial), we can define the $k$-generalized operator in a similar manner as defined by D. R. Adams \cite{MR0960950}:
\begin{equation*}
\nabla_L^ku=\left\{\begin{array}{ll}
L_{\theta,\gamma}^{\frac{k}2}u,&\mbox{if }k\mbox{ is even},  \\
(L_{\theta,\gamma}^{\frac{k-1}2}u)',&\mbox{if }k\mbox{ is odd}.
\end{array}\right.
\end{equation*}

Let $0<R<\infty$. We define the weighted Sobolev space with Navier boundary condition, denoted by $X^{k,p}_{\mathcal N,L,R}$, as follows:
\begin{equation*}
X^{k,p}_{\mathcal N,L,R}=\left\{u\in X_R^{k,p}(\alpha_0,\ldots,\alpha_k)\colon L_{\theta,\gamma}^j u(R)=0\quad \forall j=0,\ldots,\left\lfloor\frac{k-1}2\right\rfloor\right\},
\end{equation*}
where $\lfloor x\rfloor$ denotes the floor function defined as the largest integer less than or equal to $x\in\mathbb R$. When $R=\infty$, $X^{k,p}_{\mathcal N,L,\infty}$ denotes the space $X^{k,p}_\infty$. In our first theorem, we establish the equivalence between the norms $\|\nabla^k_L\cdot\|_{L^p_{\nu}}$ and $\|\cdot\|_{X^{k,p}_R}$ on $X^{k,p}_{\mathcal N,L,R}$ under the following condition:

\begin{equation}\label{hipthetagamma}
\left\{\begin{array}{ll}
     \gamma-1+\lfloor\frac{k-2}{2}\rfloor(\gamma-\theta-2)-\dfrac{\alpha_k-kp+1}{p}>0,&\mbox{if }\theta+2\geq\gamma,  \\
     \gamma-1-\dfrac{\alpha_k-kp+1}{p}>0,&\mbox{if }\theta+2<\gamma.
\end{array}\right.
\end{equation}

\begin{theo}\label{propequivnormkgrad}
Assume $0<R\leq\infty$, $k\geq2$, $\alpha_k-(k-1)p+1\geq0$, \eqref{hipthetagamma}, and $\alpha_i\geq\alpha_k-(k-i)p$ for all $i=0,\ldots,k$. The norm $\|\nabla^k_L\cdot\|_{L^p_{\nu}}$ is equivalent to the norm $\|\cdot\|_{X^{k,p}_R}$ on $X^{k,p}_{\mathcal N,L,R}(\alpha_0,\ldots,\alpha_k)$, where $\nu=\alpha_k+\lfloor\frac{k}{2}\rfloor(\theta-\gamma)p$.
\end{theo}

The objective of this paper is to provide a comprehensive study to obtain a necessary and sufficient condition on $\beta$ and $q$ for the following inequality to hold:
\begin{equation*}
\int_0^\infty\dfrac{\exp_p(\beta|u|^{\frac{p}{p-1}})}{(1+|u|)^{q}}r^\eta\mathrm dr\leq C\|u\|^p_{L^p_\eta}
\end{equation*}
for $u\in X^{k,p}_\infty$ satisfying $\|\nabla^k_Lu\|_{L^p_\nu}\leq1$, where $\alpha_k=kp-1$, $\nu=\alpha_k+\lfloor\frac{k}2\rfloor(\theta-\gamma)p$, and
\begin{equation*}
    \exp_p(t):=\sum_{j=0}^\infty\dfrac{t^{p-1+j}}{\Gamma(p+j)}.
\end{equation*}
For more details about the choice of $\exp_p$ instead of $\exp_{\lceil p\rceil}$, see \cite[Remark 1.2]{arXiv:2306.00194}.

Firstly, we provide Trudinger-Moser and Adams inequalities with the exact growth condition on weighted Sobolev spaces, precisely for the case of first order derivatives and for the case of second order derivatives.

\begin{theo}\label{theok1}
Let $p>1$, and consider $X^{1,p}_\infty(\alpha_0,\alpha_1)$ with $\alpha_1=p-1$ and $\alpha_0\geq-1$. If $\eta>-1$, then there exists a constant $C=C(p,\eta)>0$ such that for all $u\in X^{1,p}_\infty$ satisfying $\|u'\|_{L^p_{\alpha_1}}\leq 1$,
\begin{equation*}
\int_0^\infty\dfrac{\exp_p(\beta_{0,1}|u|^{\frac{p}{p-1}})}{(1+|u|)^{\frac{p}{p-1}}}r^{\eta}\mathrm dr\leq C\|u\|^{p}_{L^{p}_{\eta}},
\end{equation*}
where $\beta_{0,1}=\eta+1$.
\end{theo}

\begin{theo}\label{theo1}
Let $p>1$, and consider $X^{2,p}_\infty(\alpha_0,\alpha_1,\alpha_2)$ with $\alpha_2=2p-1$, $\alpha_1\geq p-1$, and $\alpha_0\geq\alpha_1-p$. If $\eta>-1$, $\gamma=(\alpha_2+(p-1)\eta)/p$, and $\theta\in\mathbb R$, then there exists a constant $C=C(p,\eta)>0$ such that for all $u\in X^{2,p}_\infty$ satisfying $\|L_{\theta,\gamma}u\|_{L^p_{\alpha_2+(\theta-\gamma)p}}\leq 1$,
\begin{equation*}
\int_0^\infty\dfrac{\exp_p(\beta_{0,2}|u|^{\frac{p}{p-1}})}{(1+|u|)^{\frac{p}{p-1}}}r^{\eta}\mathrm dr\leq C\|u\|^{p}_{L^{p}_{\eta}},
\end{equation*}
where $\beta_{0,2}=(\eta+1)\left(\gamma-1\right)^{\frac{p}{p-1}}$.
\end{theo}

Using Theorem \ref{theo1} we are able to establish an Adams inequality with the exact growth condition for higher order derivatives, as present in Theorem \ref{theoexactk} below. Furthermore, in Theorem \ref{theosuperc}, we demonstrate the sharpness of the constant $\beta_{0,k}$ and the power $\frac{p}{p-1}$ in the denominator.

\begin{theo}\label{theoexactk}
Consider $X^{k,p}_\infty(\alpha_0,\ldots,\alpha_k)$ with $\alpha_k-kp+1=0$. Assume $\alpha_i\geq\alpha_k-(k-i)p$ for all $i=1,\ldots,k$. Let $p>1$, $\eta>-1$, $\theta+2>\gamma$, $\theta>\lfloor\frac{k}{2}\rfloor(\theta+2-\gamma)-1$, and $\gamma=(2p-1+(p-1)\eta)/p$. There exists a constant $C=C(k,p,\eta)>0$ such that for all $u\in X^{k,p}_\infty$ satisfying $\|\nabla_L^ku\|_{L^p_{\nu}}\leq1$ (where $\nu=\alpha_k+\lfloor\frac{k}{2}\rfloor(\theta-\gamma)p$),
\begin{equation*}
\int_0^\infty\dfrac{\exp_p(\beta_{0,k}|u|^{\frac{p}{p-1}})}{(1+|u|)^{\frac{p}{p-1}}}r^\eta\mathrm dr\leq C\|u\|^p_{L^p_\eta},
\end{equation*}
where
\begin{equation*}
\beta_{0,k}=\left\{\begin{array}{ll}
     (\eta+1)\left[(\gamma-1)(\theta+2-\gamma)^{k-2}\dfrac{\Gamma\left(\frac{k}2\right)\Gamma\left(\frac{\gamma-1}{\theta+2-\gamma}\right)}{\Gamma\left(\frac{\gamma-1}{\theta+2-\gamma}-\frac{k-2}{2}\right)}\right]^{\frac{p}{p-1}}&\mbox{if }k\mbox{ is even},  \\
     (\eta+1)\left[(\gamma-1)(\theta+2-\gamma)^{k-2}\dfrac{\Gamma\left(\frac{k+1}2\right)\Gamma\left(\frac{\gamma-1}{\theta+2-\gamma}\right)}{\Gamma\left(\frac{\gamma-1}{\theta+2-\gamma}-\frac{k-3}{2}\right)}\right]^{\frac{p}{p-1}}&\mbox{if }k\mbox{ is odd}.
\end{array}\right.
\end{equation*}
\end{theo}

We note that Theorem \ref{theoexactk} is a similar result to Theorem \ref{theoexacthigher} but, instead of involving the classical Sobolev space $W^{k,\frac{N}k}(\mathbb R^N)$, it involves the weighted Sobolev space $X^{k,p}_\infty$.

We develop two corollaries that adapt the results of Theorem \ref{theoat} and Theorem \ref{theolr} to the context of weighted Sobolev spaces. Note that Corollary \ref{cor1} generalizes \cite[Theorem 1.1]{MR3670473}. 
\begin{cor}\label{cor1}
Under the assumptions of Theorem \ref{theoexactk}, for any $\beta<\beta_{0,k}$, there exists $C_\beta=C(\beta,k,p,\eta)>0$ such that
\begin{equation*}
\int_0^\infty\exp_p(\beta|u|^{\frac{p}{p-1}})r^\eta\mathrm dr\leq C_\beta\|u\|^p_{L^p_\eta}.
\end{equation*}
\end{cor}
\begin{cor}\label{cor2}
Under the assumptions of Theorem \ref{theoexactk} and $p\geq2$, for any $\tau>0$, there exists a constant $C_\tau>0$ such that
\begin{equation*}
\sup_{\underset{\|\nabla^k_Lu\|^p_{L^p_\nu}+\tau\|u\|^p_{L^p_\eta}\leq1}{u\in X^{k,p}_\infty}}\int_0^\infty\exp_p(\beta_{0,k}|u|^{\frac{p}{p-1}})r^\eta\mathrm dr\leq C_\tau.
\end{equation*}
\end{cor}
\begin{theo}\label{theosuperc}
Consider $X^{k,p}_\infty(\alpha_0,\ldots,\alpha_k)$ with $\alpha_k-kp+1=0$ and $p>1$. Then there exists a sequence $(u_n)$ such that $\|\nabla^k_Lu_n\|_{L^p_{\alpha_k+\lfloor\frac{k}{2}\rfloor(\theta-\gamma)p}}\leq 1$ and
\begin{equation}\label{eqsuc1}
    \dfrac{1}{\|u_n\|^p_{L^p_{\eta}}}\int_0^\infty\dfrac{\exp_p\left(\beta_{0,k}|u_n|^{\frac{p}{p-1}}\right)}{(1+|u_n|)^q}r^\eta\mathrm dr\overset{n\to\infty}\longrightarrow \infty,\quad\forall q<\frac{p}{p-1}.
\end{equation}
Moreover, for all $\beta>\beta_{0,k}$ and $q\geq0$, we have
\begin{equation}\label{eqsuc2}
\int_0^\infty\dfrac{\exp_p\left(\beta|u_n|^{\frac{p}{p-1}}\right)}{(1+|u_n|)^q}r^\eta\mathrm dr\overset{n\to\infty}\longrightarrow \infty.
\end{equation}
\end{theo}

\subsection{Application}
By using Theorem~\ref{theo1} in combination with the minimax argument, we obtain the existence and regularity of solutions for a class of the fourth order problem. Precisely,
\begin{theo}\label{theoapp}
Assume $\eta>1$ and $f\colon[0,\infty)\times\mathbb R\to\mathbb R$ such that $f(r,\cdot)$ is an odd function with $f(r,t)\geq0$ for all $t\geq0$. Let $\theta=(\eta+3)/2$ and $F(r,t):=\int_0^sf(r,t)\mathrm dt$ satisfying \eqref{h1} and \eqref{h2}. Then there exists $u_0\in C^4(0,\infty)\cap C^3([0,\infty))$ a nontrivial classical solution of
\begin{equation}\label{eqprobleml}
    \left\{\begin{array}{l}
\Delta^2_{\theta}u=\lambda^{-1}f(r,u)r^{\eta-\theta}\mbox{ in }(0,\infty),\\
u'(0)=(\Delta_\theta u)'(0)=0,
    \end{array}\right.
\end{equation}
with $\lambda=\int_0^\infty f(r,u_0)u_0r^\eta\mathrm dr$ and $\Delta_\theta=L_{\theta,\theta}$. Moreover, $\Delta_\theta u_0\in C^2(0,\infty)\cap C^1([0,\infty))$, $u_0''(0)=-\Delta_\theta u_0(0)/(\theta+1)$, and $u_0'''(0)=\lim_{r\to\infty}u_0(r)=\lim_{r\to\infty}\Delta_\theta u_0(r)=0$.
\end{theo}
As an immediate consequence, when $\theta=\eta=3$, we obtain the following corollary
\begin{cor}
Consider the dimensional $N=4$. Assume $f\colon\mathbb R^4\times\mathbb R\to\mathbb R$ radial on $x$ such that $f(x,\cdot)$ is an odd function with $f(x,t)\geq0$ for all $t\geq0$. Let $F(x,t):=\int_0^sf(x,t)\mathrm dt$ satisfying \eqref{h1} and \eqref{h2}. Then there exists $u_0\in C^4(\mathbb R^4)$ a nontrivial classical solution of
\begin{equation*}
\Delta^2u=\lambda^{-1}f(x,u)\mbox{ in }\mathbb R^4,
\end{equation*}
with $\lambda=\int_{\mathbb R^N} f(x,u_0)u_0\mathrm dx/|\mathbb S^{3}|$. Moreover, $\lim_{r\to\infty}u_0(r)=\lim_{r\to\infty}\Delta u_0(r)=0$.
\end{cor}

\subsection{Organization of the Paper}

In Section \ref{pre}, we introduce some preliminaries regarding the weighted Sobolev space $X^{k,p}_R$. We will establish in Section \ref{secequi} the equivalence between the norms $\|\nabla^k_L\cdot\|_{L^p_\nu}$ and $\|\cdot\|_{X^{k,p}_R}$ on $X^{k,p}_{\mathcal N,L,R}$. Moving on to Section \ref{symmandlemma}, we work on the half $\mu_{\eta,\nu}$-symmetrization, which generalizes the one considered by E. Abreu and L. Fernandez \cite{MR4097244}. Moreover, in this same section, we establish a crucial lemma (Lemma \ref{lemmakey}) that is decisive to the proof of our Theorem \ref{theo1}. To present the main result for the first and second order cases, we delve into Section \ref{mainresults}, wherein we prove Theorems \ref{theok1} and \ref{theo1}. In Section \ref{sectionk}, we present the conclusive result of the higher order Adams' inequality with exact growth on weighted Sobolev spaces (Theorem \ref{theoexactk}) and its corollaries. We obtain, in Section \ref{sectionsuperc}, the sharpness of the constants $\beta_{0,k}$ and the exponent $p/(p-1)$ in Theorem \ref{theosuperc}. Finally, we provide the proof for the application given by Theorem \ref{theoapp} in Section \ref{secapp}.

\section{Preliminaries}\label{pre}

We begin by recalling some definitions following the notation of \cite{MR1069756}. We say that a function $u\in AC_{\mathrm{loc}}(0,R)$ belongs to class $AC_L(0,R)$ if $\lim_{r\to0}u(r)=0$. Analogously, $u\in AC_{\mathrm{loc}}(0,R)$ belongs to $AC_R(0,R)$ if $\lim_{r\to R}u(r)=0$. In \cite[Example 6.8]{MR1069756}, the following Hardy-type inequality was proved:
\begin{prop}\label{prop21JMBO}
Given $p,q\in[1,\infty)$ and $\theta,\alpha\in\mathbb R$ the inequality
\begin{equation*}
\left(\int_0^R|u|^qr^\theta\mathrm dr\right)^{\frac{1}{q}}\leq C\left(\int_0^R|u'|^pr^\alpha\mathrm dr\right)^{\frac{1}{p}}
\end{equation*}
holds for some constant $C=C(p,q,\theta,\alpha,R)>0$ under the following conditions:
\begin{flushleft}
    $\mathrm{(i)}$ for $u\in AC_L(0,R)$ if and only if one of the following two conditions are fulfilled:
    \begin{enumerate}
        \item[$\mathrm{(a)}$] $1\leq p\leq q<\infty$, $q\geq\frac{(\theta+1)p}{\alpha-p+1}$, and $\alpha-p+1<0$.
        \item[$\mathrm{(b)}$] $1\leq q<p<\infty$, $q>\frac{(\theta+1)p}{\alpha-p+1}$, and $\alpha-p+1<0$.
    \end{enumerate}
    $\mathrm{(ii)}$ for $u\in AC_R(0,R)$ if and only if one of the following two conditions is fulfilled:\\
        \begin{enumerate}
            \item[$\mathrm{(a)}$] If $1\leq p\leq q<\infty$, then
        \begin{equation*}
        1\leq p\leq q\leq\dfrac{(\theta+1)p}{\alpha-p+1}\mbox{ and }\alpha-p+1>0,
        \end{equation*}
        or
        \begin{equation*}
        \theta>-1\mbox{ and }\alpha-p+1\leq0.
        \end{equation*}
        \item[$\mathrm{(b)}$] If $1\leq q<p$, then
        \begin{equation*}
        1\leq q<p<\infty,\mbox{ with }q<\dfrac{(\theta+1)p}{\alpha-p+1}\mbox{ and }\alpha-p+1>0,
        \end{equation*}
        or
        \begin{equation*}
        \theta>-1\mbox{ and }\alpha-p+1\leq0.
        \end{equation*}
\end{enumerate}
\end{flushleft}
\end{prop}
We also enunciate the case when $R=\infty$ given by \cite[Example 6.7]{MR1069756}:
\begin{prop}\label{prop21JMBOinfty}
    Given $1\leq p\leq q<\infty$ and $\theta,\alpha\in\mathbb R$ the inequality
    \begin{equation*}
\left(\int_0^\infty|u|^qr^\theta\mathrm dr\right)^{\frac{1}{q}}\leq C\left(\int_0^\infty|u'|^pr^\alpha\mathrm dr\right)^{\frac{1}{p}}
\end{equation*}
holds for some constant $C=C(p,q,\theta,\alpha)>0$ under the following conditions:
\begin{flushleft}
    $\mathrm{(i)}$ for $u\in AC_L(0,R)$ if and only if $\alpha<p-1$ and $\theta=\alpha\frac{q}p-\frac{q(p-1)}p-1$;\\
    $\mathrm{(ii)}$ for $u\in AC_R(0,R)$ if and only if $\alpha>p-1$ and $\theta=\alpha\frac{q}p-\frac{q(p-1)}p-1$.
        \end{flushleft}
\end{prop}

Now we recall some previous embeddings and radial lemmas obtained for $X^{k,p}_R$ with $R<\infty$ in \cite{arXiv:2302.02262} and for $X^{k,p}_\infty$ in \cite{arXiv:2306.00194}. Moreover, we also present a density result in $X^{k,p}_\infty$.

\begin{theo}[Theorem 1.1 in \cite{arXiv:2302.02262}]\label{theo32}
Let $p\geq1$, $0<R<\infty$, and $\theta\geq\alpha_k-kp$.
\begin{flushleft}
    $\mathrm{(a)}$ (Sobolev case) If $ \alpha_k-kp+1>0$, then the continuous embedding holds
\begin{equation*}X^{k,p}_R(\alpha_0,\ldots,\alpha_k)\hookrightarrow L^q_\theta(0,R)\quad \text{for all}\quad 1\leq q\leq p^*:=\dfrac{(\theta+1)p}{\alpha_k-kp+1}. \end{equation*}
Moreover, the embedding is compact if $q<p^*.$ \\
    $\mathrm{(b)}$ (Sobolev Limit case) If $\alpha_k-kp+1=0$, then the compact embedding holds
\begin{equation*}
    X^{k,p}_R(\alpha_0,\ldots,\alpha_k)\hookrightarrow L^q_\theta(0,R)\quad \text{for all} \quad q\in[1,\infty).
\end{equation*}\\
$\mathrm{(c)}$ (Morrey case) If $\alpha_k-kp+1<0, \; p>1, \; \alpha_k\geq0$
then the continuous embedding holds
\begin{equation*}
   X^{k,p}_R(\alpha_0,\ldots,\alpha_k)\hookrightarrow C^{k-\lfloor\frac{\alpha_k+1}{p}\rfloor-1,\gamma}([0,R]),
\end{equation*}
where $\gamma=\min\left\{1+\left\lfloor\frac{\alpha_k+1}{p}\right\rfloor-\frac{\alpha_k+1}{p},1-\frac1p\right\}$ if $\frac{\alpha_k+1}{p}\notin\mathbb Z$
and $\gamma \in (0,1)$ if $\frac{\alpha_k+1}{p}\in\mathbb Z$.
        \end{flushleft}
\end{theo}

\begin{lemma}[Propositions 2.3, 2.4 and 2.5 in \cite{arXiv:2302.02262}]
Let $X^{k,p}_R(\alpha_0,\ldots,\alpha_k)$ be the weighted Sobolev space with $0<R<\infty$ and $1\leq p<\infty$. For all $u\in X^{k,p}_R$ and $r\in(0,R]$ we have
\begin{equation}\label{RL1}
|u(r)|\leq C\dfrac1{r^{\frac{\alpha_k-kp+1}{p}}}\|u\|_{X_R^{k,p}},\mbox{ if }\alpha_k-kp+1>0;
\end{equation}
and
\begin{equation}\label{RL2}
|u(r)|\leq C|\log (r/R)|^{\frac{p-1}p}\|u\|_{X^{k,p}_R},\mbox{ if }\alpha_k-kp+1=0\mbox{ and }p>1,
\end{equation}
where $C=C(\alpha_0,\ldots,\alpha_{k},p,k,R)>0$ is a constant. Moreover, if $\alpha_k-kp+1=0$ and $p=1$, then
\begin{equation}\label{RL3}
    X^{k,p}_R\hookrightarrow C([0,R]).
\end{equation}
\end{lemma}

\begin{theo}[Theorem 1.1 in \cite{	arXiv:2306.00194}]\label{theoimersaoinfinito}
Let the weighted Sobolev space $X^{k,p}_\infty(\alpha_0,\ldots,\alpha_k)$ such that $p\geq1$ and $\alpha_{j}\geq\alpha_k-(k-j)p$ for all $j=0,\ldots,k-1$. Consider $\alpha_k-kp\leq\theta\leq\alpha_0$.
\begin{flushleft}
    $\mathrm{(a)}$ (Sobolev case) If $\alpha_k-kp+1>0$, then the following continuous embedding holds:
    \begin{equation*}
    X^{k,p}_\infty\hookrightarrow L^q_\theta,\quad p\leq q\leq p^*,
    \end{equation*}
    where
    \begin{equation*}
    p^*:=\dfrac{(\theta+1)p}{\alpha_k-kp+1}.
    \end{equation*}
    Moreover, it is a compact embedding if one of the following two conditions is fulfilled:
    \begin{enumerate}
        \item[$\mathrm{(i)}$] $\theta=\alpha_0$ and $p<q<p^*$;
        \item[$\mathrm{(ii)}$] $\theta<\alpha_0$ and $p\leq q<p^*$.
    \end{enumerate}
    $\mathrm{(b)}$ (Sobolev Limit case) If $\alpha_k-kp+1=0$, then the following continuous embedding holds:
    \begin{equation*}
    X^{k,p}_\infty\hookrightarrow L^q_\theta,\quad p\leq q<\infty.
    \end{equation*}
    Moreover, it is compact embedding if one of the following two conditions is fulfilled:
    \begin{enumerate}
        \item[$\mathrm{(i)}$] $\theta=\alpha_0$, $q>p>1$ and $\alpha_0>-1$;
        \item[$\mathrm{(ii)}$] $\theta<\alpha_0$ and $q\geq p\geq1$.
    \end{enumerate}
\end{flushleft}
\end{theo}

\begin{lemma}[Lemma 2.1 in \cite{arXiv:2306.00194}]\label{lemmacompsuppdense}
The space $\Upsilon:=\{u|_{[0,\infty)} : u \in C_0^\infty(\mathbb{R})\}$ is dense in the weighted Sobolev space $X^{k,p}_\infty$ under the assumptions
$p\geq1$ and $\alpha_{j-1}\geq\alpha_j-p$ for all $j=1,\ldots,k$.
\end{lemma}

\section{Equivalence between  \texorpdfstring{$\|\nabla_L^k\cdot\|_{L^p_\nu}$}{} and \texorpdfstring{$\|\cdot\|_{X^{k,p}_R}$}{} on  \texorpdfstring{$X^{k,p}_{\mathcal N,L,R}$}{}}\label{secequi}

Now we focus on proving Theorem \ref{propequivnormkgrad}, which guarantees the equivalence between the norms $\|\nabla^k_L\cdot\|_{L^p_\nu}$ and $\|\cdot\|_{X^{k,p}_R}$ on $X^{k,p}_{\mathcal N,L,R}$ with $\nu=\alpha_k+\lfloor\frac{k}2\rfloor(\theta-\gamma)p$. Before we present the proof, we need the following four lemmas. The last one (Lemma \ref{qww}) is a regularity result for the weighted Sobolev space, which, under some circumstances, implies $L_{\theta,\gamma}u\in X^{k,p}_R\Rightarrow u\in X^{k+2,p}_R$.

\begin{lemma}\label{lemmajaosn}
Let $0<R\leq\infty$, $j=0,\ldots,k$, and $X^{k,p}_R(\alpha_0,\ldots,\alpha_k)$ be the weighted Sobolev space such that $\alpha_i-(i-1)p+1\geq0$ for all $i=j,\ldots,k$. In the case of $R=\infty$ we also assume $\alpha_{\ell-i}\geq\alpha_\ell-ip$ for all $\ell=j,\ldots,k$ and $i=0,\ldots,\ell-1$. For each $u\in X^{k,p}_R$, we have $\nabla^j_L u\in X^{k-j,p}_R(\alpha_j+\lfloor\frac{j}{2}\rfloor(\theta-\gamma)p,\ldots,\alpha_k+\lfloor\frac{j}{2}\rfloor(\theta-\gamma)p)$ with
\begin{equation*}
\|\nabla_L^ju\|_{X^{k-j,p}_R}\leq C\|u\|_{X^{k,p}_R},
\end{equation*}
where $C>0$ does not depend on $u$.
\end{lemma}
\begin{proof}
We can assume $j\geq2$; otherwise, the lemma is trivial. By induction on $j$ and $\ell$, we can prove that
\begin{equation}\label{doubleinduction}
\left(\nabla^j_L u\right)^{(\ell)}(r)=\sum_{i=0}^{j+\ell-1}C_{ij\ell}u^{(j+\ell-i)}r^{\lfloor\frac{j}{2}\rfloor(\gamma-\theta)-i},\quad\forall \ell=0,\ldots,k-j,
\end{equation}
for some $C_{ij\ell}=C_{ij\ell}(\theta,\gamma)\in\mathbb R$. It is sufficient to check that
\begin{equation}\label{eqbabubabu}
\|(\nabla^j_Lu)^{(\ell)}\|_{L^p_{\nu}}\leq C\|u\|_{X^{k,p}_R},
\end{equation}
where $\nu=\alpha_{j+\ell}+\lfloor\frac{j}{2}\rfloor(\theta-\gamma)p$. Since \eqref{doubleinduction} and $\alpha_{j+\ell}-(j+\ell-1)p+1\geq0$ hold, we apply Theorem \ref{theo32} (on $u^{(j+\ell-i)}\in X^{i,p}_R(\alpha_{j+\ell-i},\ldots,\alpha_{j+\ell})$) for $0<R<\infty$ and Theorem \ref{theoimersaoinfinito} for $R=\infty$ to obtain
\begin{equation*}
\|(\nabla^j_L u)^{(\ell)}\|_{L^p_{\nu}}^p\leq C\sum_{i=0}^{j+\ell-1}\int_0^R\left|u^{(j+\ell-i)}\right|^pr^{\alpha_{j+\ell}-ip}\mathrm dr\leq C\sum_{i=0}^{j+\ell-1}\|u^{(j+\ell-i)}\|^p_{X^{i,p}_R}\leq C\|u\|_{X^{j+\ell,p}_R}^p.
\end{equation*}
This concludes \eqref{eqbabubabu} and the proof of the lemma.
\end{proof}

\begin{lemma}\label{corxnrbanach}
$X^{k,p}_{\mathcal N,L,R}(\alpha_0,\ldots,\alpha_k)$ is a Banach space with the norm $\|\cdot\|_{X^{k,p}_R}$.
\end{lemma}
\begin{proof}
We only need to consider the case when $0<R<\infty$. Let $(u_n)$ be a Cauchy sequence in $X^{k,p}_{\mathcal N,L,R}$. Then $u_n\to u$ in $X^{k,p}_R$ for some $u\in X^{k,p}_R$. By \eqref{doubleinduction},
\begin{equation*}
|L_{\theta,\gamma}^ju(R)|=|\nabla^{2j}_Lu_n(R)-\nabla^{2j}_{L}u(R)|\leq \sum_{i=0}^{2j-1}C_{ij}|u_n^{(2j-i)}(R)-u^{(2j-i)}(R)|R^{j(\gamma-\theta)-i},
\end{equation*}
for all $0\leq j\leq\lfloor\frac{k-1}{2}\rfloor$ and $n\in\mathbb N$. Using \cite[Lemma 2.1]{arXiv:2302.02262} we conclude
\begin{equation*}
|L_{\theta,\gamma}^ju(R)|\leq \sum_{i=0}^{2j-1}C_{ij}\widetilde C_{ij}\|u_n^{(2j-i)}-u^{(2j-i)}\|_{X^{1,p}_R}R^{j(\gamma-\theta)-i}\overset{n\to\infty}\longrightarrow 0.
\end{equation*}
Therefore, we proved that $u\in X^{k,p}_{\mathcal N,L,R}$ which concludes the completeness of $X^{k,p}_{\mathcal N,L,R}$.
\end{proof}

\begin{lemma}\label{lemmafjs}
Let $\alpha_2\in\mathbb R$ and $v\in L^p_{\alpha_2}(0,R)$ with $0<R\leq\infty$ and $1\leq p<\infty$. Define
\begin{equation}\label{expressionu}
u(r)=\int_r^Rt^{-\gamma}\int_0^tv(s)s^\theta\mathrm ds\mathrm dt\quad\mathrm{a.e.}\ r\in(0,R).
\end{equation}
If $\theta>(\alpha_2-p+1)/p$, then $u\in X^{2,p}_R(\alpha_2+(\gamma-\theta-2)p,\alpha_2+(\gamma-\theta-1)p,\alpha_2+(\gamma-\theta)p)$ with $\|u\|_{X^{2,p}_R}\leq C\|v\|_{L^p_{\alpha_2}}$, where $C=C(\theta,\gamma,\alpha_2,p,R)>0$.
\end{lemma}
\begin{proof}
The case $R<\infty$ is a direct consequence of \cite[Lemma 4.1]{arXiv:2302.02262} with $\widetilde\alpha_2=\alpha_2+(\gamma-\theta)p$. Note that
\begin{equation*}
    \theta>\frac{\alpha_2-p+1}p\Leftrightarrow\gamma>\frac{\widetilde\alpha_2-p+1}p.
\end{equation*}
For the case $R=\infty$, the same argument as \cite[Lemma 4.1]{arXiv:2302.02262} follows using Proposition \ref{prop21JMBOinfty} instead of Proposition \ref{prop21JMBO}.
\end{proof}

\begin{lemma}\label{qww}
Let $v\in X^{k,p}_R(\alpha_k-kp,\ldots,\alpha_k)$ and
\begin{equation*}
u(r):=\int_r^Rt^{-\gamma}\int_0^tv(s)s^{\theta}\mathrm ds\mathrm dt,\quad r\in(0,R),
\end{equation*}
where $\theta>(\alpha_k-kp-p+1)/p$. Then
\begin{equation*}
u\in X^{k+2,p}_R(\alpha_k+(\gamma-\theta-k-2)p,\alpha_k+(\gamma-\theta-k-1)p,\ldots,\alpha_k+(\gamma-\theta)p).
\end{equation*}
\end{lemma}
\begin{proof}
We claim that
\begin{equation}\label{qw1}
v(r)r^{\theta-\gamma-k+i}\in X^{i,p}_R(\alpha_k+(\gamma-\theta-i)p,\ldots,\alpha_k+(\gamma-\theta)p),\quad\forall i=0,\ldots,k.
\end{equation}
Indeed, the case $i=0$ in \eqref{qw1} follows by $v\in L^p_{\alpha_k-kp}$. Suppose \eqref{qw1} holds for any $j=0,\ldots,i$. Our task is to prove that
\begin{equation*}
v(r)r^{\theta-\gamma-k+i+1}\in X^{i+1,p}_R(\alpha_k+(\gamma-\theta-i-1)p,\ldots,\alpha_k+(\gamma-\theta)p).
\end{equation*}
Since $v(r)r^{\theta-\gamma-k+i+1}\in L^p_{\alpha_k+(\gamma-\theta-i-1)p}$, we only need to check that
\begin{equation*}
v'(r)r^{\theta-\gamma-k+i+1}+(\theta-\gamma-k+i+1)v(r)r^{\theta-\gamma-k+i}=\left(v(r)r^{\theta-\gamma-k+i+1}\right)'\in X^{i,p}_R.
\end{equation*}
Using induction hypothesis on $v'\in X^{k-1,p}_R$ and $v\in X^{k,p}_R$ we conclude \eqref{qw1}.

Let us verify that
\begin{equation}\label{qw2}
r^{-(\gamma+i)}\int_0^rv(s)s^{\theta}\mathrm ds\in L^p_{\alpha_k+(\gamma-\theta+i-k-1)p},\quad\forall i\in\mathbb Z.
\end{equation}
Denote $w(r)=\int_0^rv(s)s^{\theta}\mathrm ds$. By $\alpha_k-(\theta+k+1)p+1<0$ with Propositions \ref{prop21JMBO} and \ref{prop21JMBOinfty},
\begin{align*}
\int_0^R\left|r^{-(\gamma+i)}\int_0^rv(s)s^{\theta}\mathrm ds\right|^pr^{\alpha_k+(\gamma-\theta+i-k-1)p}\mathrm dr&=\int_0^R|w(r)|^pr^{\alpha_k-(\theta+k+1)p}\mathrm dr\\
&\leq C\int_0^R|w'(r)|^pr^{\alpha_k-(\theta+k)p}\mathrm dr\\
&=C\|v\|^p_{L^p_{\alpha_k-kp}},\quad\forall i\in\mathbb Z.
\end{align*}

We already have $u\in X^{2,p}_R(\alpha_k+(\gamma-\theta-k-2)p,\alpha_k+(\gamma-\theta-k-1)p,\alpha_k+(\gamma-\theta-k)p)$, by Lemma \ref{lemmafjs}. Then it is enough to prove that
\begin{equation*}
u''(r)=\gamma r^{-(\gamma+1)}\int_0^rv(s)s^{\theta}\mathrm ds-v(r)r^{\theta-\gamma}\in X^{k,p}_R(\alpha_k+(\gamma-\theta-k)p,\ldots,\alpha_k+(\gamma-\theta)p).
\end{equation*}
In view of \eqref{qw1}, the proof is completed by showing that
\begin{equation}\label{qw3}
r^{-(\gamma+1)}\int_0^rv(s)s^{\theta}\mathrm ds\in X^{k,p}_R(\alpha_k+(\gamma-\theta-k)p,\ldots,\alpha_k+(\gamma-\theta)p).
\end{equation}
Indeed, for each $i\in\mathbb N\cup\{0\}$, set
\begin{equation*}
w_i(r):=r^{-(\gamma+i+1)}\int_0^rv(s)s^{\theta}\mathrm ds.
\end{equation*}
\eqref{qw2} guarantees $w_i\in L^p_{\alpha_k+(\gamma-\theta+i-k)p}$. Note that, using \eqref{qw1},
\begin{align*}
w_0\in X^{k,p}_R&\Leftrightarrow-(\gamma+1)r^{-(\gamma+2)}\int_0^rv(s)s^{\theta}\mathrm ds+v(r)r^{\theta-\gamma-1}=w_0'\in X^{k-1,p}_R\\
&\Leftrightarrow w_1\in X^{k-1,p}_R\\
&\Leftrightarrow w_1'\in X^{k-2,p}_R\\
&\ \ \vdots\\
&\Leftrightarrow w_k\in L^p_{\alpha_k+(\gamma-\theta)p}.
\end{align*}
This concludes \eqref{qw3} and therefore the lemma.
\end{proof}

\begin{proof}[Proof of Theorem \ref{propequivnormkgrad}]
Firstly, note that we can rewrite \eqref{hipthetagamma} as follows:
\begin{equation*}
\left\{\begin{array}{ll}
     \theta+1+\lfloor\frac{k}2\rfloor(\gamma-\theta-2)-\dfrac{\alpha_k-kp+1}{p}>0,&\mbox{if }\theta+2\geq\gamma,  \\
\theta+1+(\gamma-\theta-2)-\dfrac{\alpha_k-kp+1}{p}>0,&\mbox{if }\theta+2<\gamma.
\end{array}\right.
\end{equation*}
Therefore, we have
\begin{equation}\label{hipgammatheta}
\theta>\dfrac{\alpha_k-kp+i(\theta+2-\gamma)p-p+1}p,\quad\forall i=1,\ldots,\left\lfloor\frac{k}2\right\rfloor.
\end{equation}

As did in the proof of \cite[Proposition 4.1]{arXiv:2302.02262}, our task is to prove that the mapping
\begin{eqnarray*}
    \phi\colon X^{k,p}_{\mathcal N,L,R}&\longrightarrow &L^p_{\alpha_k+\lfloor\frac{k}{2}\rfloor(\theta-\gamma)p}\\
    u&\longmapsto&\nabla^k_Lu
\end{eqnarray*}
is a linear isomorphism. Note that $\phi$ is linear and, by Lemma \ref{lemmajaosn}, continuous. By the Open Mapping Theorem and Lemma \ref{corxnrbanach}, it suffices to prove that $\phi$ is bijective. Let $u\in X^{k,p}_{\mathcal N,L,R}$ with $\phi(u)=\nabla^k_{L}u=0$. We claim that
\begin{equation}\label{eqclaimhsu}
r^{\gamma}\nabla^i_{L}u\overset{r\to0}\longrightarrow0\quad\forall i\in\{1,\ldots,k-1\}\mbox{ odd number}.
\end{equation}
Indeed, by \eqref{doubleinduction}, we have
\begin{equation}\label{eq34}
\left|r^{\gamma}\nabla^i_{L}u\right|\leq\sum_{\ell=0}^{i-1}C_{i\ell}|u^{(i-\ell)}|r^{\gamma+\frac{i-1}{2}(\gamma-\theta)-\ell}.
\end{equation}
Suppose first that $\alpha_k-(k-1)p+1>0$. Since $u^{(i-\ell)}\in X^{k-i+\ell,p}_R$, the radial lemma \eqref{RL1} implies $|u^{(i-\ell)}(r)|\leq C\|u\|_{X^{k,p}_R}r^{-\frac{\alpha_k-(k-i+\ell)p+1}{p}}$. Then
\begin{equation*}
\left|r^{\gamma}\nabla^i_{L}u\right|\leq C\|u\|_{X^{k,p}_R}\sum_{\ell=0}^{i-1}r^{\gamma+\frac{i-1}{2}(\gamma-\theta)-\ell-\frac{\alpha_k-(k-i+\ell)p+1}p}\leq C\|u\|_{X^{k,p}_R}r^{\gamma-1+\frac{(i-1)(\gamma-\theta-2)}{2}-\frac{\alpha_k-kp+1}p}.
\end{equation*}
Thus, \eqref{eqclaimhsu} follows from \eqref{hipthetagamma}. For the case $\alpha_k-(k-1)p+1=0$, we only need to consider the term with $\ell=i-1$ in \eqref{eq34} because $u^{(i-\ell)}\in X^{k-i+\ell,p}_R$ fits in the Sobolev case when $\ell<i-1$ and it was solved in the other case. Using radial lemmas \eqref{RL2} and \eqref{RL3} instead of \eqref{RL1}, we obtain
\begin{equation*}
\left|u^{(i-\ell)}r^{\gamma+\frac{i-1}2(\gamma-\theta)-i+1}\right|\leq C\|u\|_{X^{k,p}_R} r^{\gamma+\frac{(i-1)(\gamma-\theta-2)}{2}}|\log (r/R)|^{\frac{p-1}p}.
\end{equation*}
By \eqref{hipthetagamma} and $\alpha_k-(k-1)p+1=0$, we have
\begin{equation*}
\left\{\begin{array}{ll}
     \gamma+\frac{(k-2)(\gamma-\theta-2)}2>0,&\mbox{if }\theta+2\geq\gamma,  \\
     \gamma>0,&\mbox{if }\theta+2<\gamma.
\end{array}\right.
\end{equation*}
Therefore, we conclude \eqref{eqclaimhsu}.

Let $j$ be the integer such that $k=2j$ or $k=2j+1$. From $\nabla^k_{L} u=0$ and $u\in X^{k,p}_{\mathcal N,L,R}$, we have $\nabla^{2j}_{L}u=L_{\theta,\gamma}^ju=0$. The conditions $u\in X^{k,p}_{\mathcal N,L,R}$ and \eqref{eqclaimhsu} guarantee that we can apply the following result $j$-times on $L_{\theta,\gamma}^ju=0$ to obtain $u=0$:
\begin{equation*}
L_{\theta,\gamma}v=0,\ r^{\gamma}v'\overset{r\to0}\longrightarrow0\mbox{ and }v(R)=0\Rightarrow v=0.
\end{equation*}
This concludes that $\phi$ is injective.

Now let us prove that $\phi$ is surjective. Initially, suppose $k=2j$. Given $v\in L^p_{\alpha_k+\lfloor\frac{k}2\rfloor(\theta-\gamma)p}$, Lemma \ref{lemmafjs} implies that there exists
\begin{equation*}
u_1\!\in\! X^{2,p}_R\!\left(\alpha_k+(j-1)(\theta-\gamma)p-2p,\alpha_k+(j-1)(\theta-\gamma)p-p,\alpha_k+(j-1)(\theta-\gamma)p\right)\cap X^{1,p}_{0,R}
\end{equation*}
such that $L_{\theta,\gamma}u_1=v$. Again using Lemma \ref{lemmafjs} we obtain
\begin{equation*}
u_2\in X^{2,p}_R(\alpha_k+(j-2)(\theta-\gamma)p-4p,\alpha_k+(j-2)(\theta-\gamma)p-3p,\alpha_k+(j-2)(\theta-\gamma)p-2p)\cap X^{1,p}_{0,R}
\end{equation*}
with $L_{\theta,\gamma}u_2=u_1$. By Lemma \ref{qww} and $u_1\in X^{1,p}_{0,R}$, we have
\begin{equation*}
u_2\in X^{4,p}_{\mathcal N,L,R}(\alpha_k+(j-2)(\theta-\gamma)p-4p,\ldots,\alpha_k+(j-2)(\theta-\gamma)p)
\end{equation*}
and $L^2_{\theta,\gamma}u_2=v$. Proceeding with this argument, we obtain $u_j\in X^{2j,p}_{\mathcal N,L,R}(\alpha_k-2jp,\ldots,\alpha_k)$ such that $\phi(u_j)=L^j_{\theta,\gamma}u_j=v$, because the hypothesis on $\theta$ (see equation \eqref{hipgammatheta}) guarantees
\begin{equation*}
\theta>\dfrac{\alpha_k-2jp+i(\theta+2-\gamma)p-p+1}{p},\quad\forall i=1,\ldots,j.
\end{equation*}

Now suppose $k=2j+1$. Given $v\in L^p_{\alpha_k+\lfloor{\frac{k}{2}\rfloor(\theta-\gamma)p}}$, set
\begin{equation*}
\widetilde u(r)=-\int_r^Rv(s)\mathrm ds.
\end{equation*}
By Proposition \ref{prop21JMBO}, $\widetilde u\in X^{1,p}_{0,R}(\alpha_k+\lfloor{\frac{k}{2}\rfloor(\theta-\gamma)p}-p,\alpha_k+\lfloor{\frac{k}{2}\rfloor(\theta-\gamma)p})$, and $\|\widetilde u\|_{L^p_{\alpha_k+\lfloor{\frac{k}{2}\rfloor(\theta-\gamma)p}-p}}\leq C\|v\|_{L^p_{\alpha_k+\lfloor{\frac{k}{2}\rfloor(\theta-\gamma)p}}}$. As in the proof for $k=2j$ with \eqref{hipgammatheta}, we obtain $u\in X^{2j+1,p}_{\mathcal N,L,R}(\alpha_k-kp,\ldots,\alpha_k)$ such that $L^j_{\theta,\gamma}u=\widetilde u$. Therefore, $\phi(u)=\nabla^k_{L}u=v$, which concludes that $\phi$ is surjective.
\end{proof}

\section{Weighted Symmetrization and Crucial Lemma}\label{symmandlemma}

\subsection{Half-Weighted Symmetrization}\label{symmandlemma1}

In this subsection, we extend the previous work in \cite{MR4097244} on a half-weighted Schwarz symmetrization using a single measure $\mu_\eta$. However, we now consider a more general scenario by incorporating two different measures: $\mu_\eta$ and $\mu_\nu$. The measure $\mu_\eta$ is defined as $\mathrm d\mu_\eta(r)=r^\eta\mathrm dr$ with $\eta>-1$. Additionally, if $M\subset\mathbb R$ is a measurable set with finite $\mu_\eta-$measure, we let $\nu>-1$ and $M^*=[0,R]$ with $R\in(0,\infty]$ such that
\begin{equation*}
\mu_\nu((0,R))=\mu_\eta(M).
\end{equation*}
Furthermore, if $u\colon M\to\mathbb R$ is a measurable function satisfying
\begin{equation*}
\mu_{\eta,u}(t):=\mu_\eta(\{x\in M\colon|u(x)|>t\})<\infty\mbox{ for all }t>0,
\end{equation*}
we denote the half-weighted Schwarz symmetrization of $u$ as $u^*_{\eta,\nu}\colon[0,R]\to\mathbb R$, or simply the half $\mu_{\eta,\nu}-$symmetrization of $u$, given by
\begin{equation*}
u^*_{\eta,\nu}(r)=\inf\{t\in\mathbb R\colon\mu_\eta(\{x\in M\colon|u(x)|>t\})<\mu_\nu(0,r)\},\quad\forall r\in(0,R),
\end{equation*}
with $u^*_{\eta,\nu}(0):=\lim_{r\to0}u^*_{\eta,\nu}(r)=\mathrm{ess\ sup}(u)$ and $u^*_{\eta,\nu}(R):=\lim_{r\to R}u^*_{\eta,\nu}(r)=0$. For simplicity, we denote $\{u>t\}$ as the set $\{x\in M\colon u(x)>t\}$. Building upon the classical results for Schwarz symmetrization, we establish the following two propositions.
\begin{prop}\label{prop20}
The function $u^*_{\eta,\nu}$ is nonincreasing and left-continuous.
\end{prop}
\begin{proof}
The monotonicity of $u^*_{\eta,\nu}$ follows directly from
\begin{equation*}
    \{t\in\mathbb R\colon\mu_\eta(\{|u|>t\})<\mu_\nu(0,r_1)\}\subset \{t\in\mathbb R\colon\mu_\eta(\{|u|>t\})<\mu_\nu(0,r_2)\},\quad\forall r_1<r_2.
\end{equation*}

Let $r\in[0,R]$. By definition of $u^*_{\eta,\nu}$, given $\varepsilon>0$, there exists a $t$ such that $u^*_{\eta,\nu}(r)\leq t<u^*_{\eta,\nu}(r)+\varepsilon$ and $\mu_\eta(\{|u|>t\})<\mu_\nu(0,r)$. Take $\delta>0$ such that $\mu_\eta(\{|u|>t\})<\mu_\nu(0,r-\delta)$. Then, for all $s\in(r-\delta,r]$, we have $\mu_\eta(\{|u|>t\})<\mu_\nu(0,s)$ and so $u^*_{\eta,\nu}(r)\leq u^*_{\eta,\nu}(s)\leq t<u^*_{\eta,\nu}(r)+\varepsilon$. This proves that $u^*_{\eta,\nu}$ is left-continuous.
\end{proof}
\begin{prop}\label{prop21}
The functions $u\colon M\to\mathbb R$ and $u^*_{\eta,\nu}\colon[0,R]\to\mathbb R$, where $\mu_\eta(M)=\mu_\nu((0,R))$, are equimeasurable with respect to $\mu_\eta$ and $\mu_{\nu}$ respectively, i.e.,
\begin{equation}\label{eqhaks}
\mu_\eta(\{x\in M\colon |u(x)|>t\})=\mu_\nu(\{x\in [0,R]\colon u^*_{\eta,\nu}(x)>t\}),\quad\forall t\in\mathbb R.
\end{equation}
\end{prop}
\begin{proof}
Let $t\in\mathbb R$. We can suppose $0\leq t\leq u^*_{\eta,\nu}(0)$. By the definition of $u^*_{\eta,\nu}$, we obtain
\begin{equation*}
\{r\colon u^*_{\eta,\nu}(r)>t\}\subset\{r\colon\mu_\eta(\{|u|>t\})\geq\mu_\nu(0,r)\}.
\end{equation*}
Set $R_0:=\sup\{r\colon u^*_{\eta,\nu}(r)>t\}$. Since $u^*_{\eta,\nu}$ is nonincreasing and left-continuous, we have $[0,R_0)\subset\{u^*_{\eta,\nu}>t\}\subset[0,R_0]$. Then
\begin{equation*}
\mu_\nu(\{u^*_{\eta,\nu}>t\})=\mu_\nu(0,R_0)\leq \mu_\eta(\{|u|>t\}).
\end{equation*}

Suppose, by contradiction, that $\mu_\nu(\{u^*_{\eta,\nu}>t\})<\mu_\eta(\{|u|>t\})$. Since $[0,R_0)\subset\{u^*_{\eta,\nu}>t\}$, we have $\mu_\nu(0,R_0)<\mu_\eta(\{|u|>t\})$. Let $\delta>0$ such that
\begin{equation*}
\mu_{\nu}(0,R_0+\delta)=\mu_\eta(\{|u|>t\}).
\end{equation*}
Thus by $\{u^*_{\eta,\nu}>t\}\subset[0,R_0]$ we obtain $u^*_{\eta,\nu}(R_0+\delta)=t$. Since $u^*_{\eta,\nu}$ is nonincreasing, $u^*_{\eta,\nu}(r)=t$ for all $r\in(R_0,R_0+\delta]$. Making $r\to R_0$ with $r>R_0$ in
\begin{equation*}
\mu_\eta(\{|u|>t\})\leq\mu_\nu(0,r),
\end{equation*}
we obtain $\mu_\eta(\{|u|>t\})\leq\mu_\nu(\{u^*_{\eta,\nu}>t\})$, which is a contradiction. Therefore the equality \eqref{eqhaks} holds.
\end{proof}
\begin{cor}\label{cor21}
    Let $u\colon M\to\mathbb R$ be measurable and $\Psi\colon \mathbb R\to[0,\infty)$ be nonnegative and measurable. Then
    \begin{equation*}
    \int_M\Psi(|u(r)|)r^\eta\mathrm dr=\int_0^{R}\Psi(u^*_{\eta,\nu}(r))r^\nu\mathrm dr,
    \end{equation*}
    where $u^*_{\eta,\nu}$ is the half $\mu_{\eta,\nu}$-symmetrization of $u$ and $R\in(0,\infty]$ is such that $\mu_\eta(M)=\mu_\nu((0,R))$.
\end{cor}
\begin{proof}
Proposition \ref{prop21} proved the result for $\Psi$ equal to the characteristic function of a set such as $(t,\infty)$. Then the corollary holds if $\Psi$ is the characteristic function of a Borel set and hence if $\Psi$ is a simple function nonnegative. By the Monotone Convergence Theorem, we conclude for all $\Psi\colon \mathbb R\to[0,\infty)$ nonnegative and measurable.
\end{proof}

Let $u^{**}_{\eta,\nu}\colon[0,\infty)\to\mathbb [0,\infty)$ be the maximal function of the rearrangement of $u^*_{\eta,\nu}$, defined as
\begin{equation}\label{plb}
u^{**}_{\eta,\nu}(t):=\dfrac{\nu+1}{t^{\nu+1}}\int_0^tu^*_{\eta,\nu}(s)s^\nu\mathrm ds.
\end{equation}
It follows that $u^{**}_{\eta,\nu}$ is also nonincreasing and $u^*_{\eta,\nu}\leq u^{**}_{\eta,\nu}$. Based on \cite[Lemma 3.9]{MR0928802}, we can estimate the norm $\|u^{**}_{\eta,\nu}\|_{L^p_\alpha}$ as stated in the following lemma.
\begin{lemma}\label{lemma21} Let $u^*_{\eta,\nu}$ be the half $\mu_{\eta,\nu}$-symmetrization of $u$. If $1<p<\infty$ and $\alpha<\nu p+p-1$, then
\begin{equation*}
\int_0^\infty|u^{**}_{\eta,\nu}|^pt^\alpha\mathrm dt\leq (\nu+1)^p\left(\dfrac{p}{p-1+\nu p-\alpha}\right)^p\int_0^\infty|u^*_{\eta,\nu}|^pt^\alpha\mathrm dr.
\end{equation*}
\end{lemma}
\begin{proof}
    Let $a=(\alpha+1-\nu p)(p-1)/p^2$ to simplify notation. By writing $u^*_{\eta,\nu}(s)=s^{-a}s^au^*_{\eta,\nu}(s)$ and applying H\"older's inequality, we obtain (note that $1-ap/(p-1)>0$ because $\alpha<\nu p+p-1$)
    \begin{align*}
    u^{**}_{\eta,\nu}(t)&\leq\dfrac{\nu+1}{t^{\nu+1}}\int_0^tu^*_{\eta,\nu}(s)s^\nu\mathrm ds\leq \dfrac{\nu+1}{t^{\nu+1}}\left(\int_0^ts^{-\frac{ap}{p-1}}\mathrm ds\right)^{\frac{p-1}p}\left(\int_0^t|u^*_{\eta,\nu}|^ps^{\nu p+ap}\mathrm ds\right)^{\frac1p}\\
    &=(\nu+1)\left(\dfrac{p}{p-1+\nu p-\alpha}\right)^{\frac{p-1}p}t^{-\frac{1}p-\nu-a}\left(\int_0^t|u^*_{\eta,\nu}|^ps^{\nu p+ap}\mathrm ds\right)^{\frac1p}.
    \end{align*}
    Thus,
    \begin{equation*}
    |u^{**}_{\eta,\nu}(t)|^p\leq(\nu+1)^p\left(\dfrac{p}{p-1+\nu p-\alpha}\right)^{p-1}t^{-1-\nu p-ap}\int_0^t|u^*_{\eta,\nu}|^ps^{\nu p+ap}\mathrm ds.
    \end{equation*}
    Using Fubini's Theorem, we get (note that $\alpha<\nu p+p-1$ implies $\alpha<\nu p+ap$)
    \begin{align*}
    \int_0^\infty|u^{**}_{\eta,\nu}|^pt^{\alpha}\mathrm dt&\leq (\nu+1)^p\left(\dfrac{p}{p-1+\nu p-\alpha}\right)^{p-1}\\
    &\qquad\int_0^\infty \int_s^\infty t^{\alpha-1-\nu p-ap}|u^*_{\eta,\nu}(s)|^ps^{\nu p+ap}\mathrm dt\mathrm ds\\
    &=(\nu+1)^p\left(\dfrac{p}{p-1+\nu p-\alpha}\right)^{p-1}\dfrac{1}{\nu p+ap-\alpha}\int_0^\infty|u^*_{\eta,\nu}(s)|^ps^{\alpha}\mathrm ds\\
    &=(\nu+1)^p\left(\dfrac{p}{p-1+\nu p-\alpha}\right)^{p}\int_0^\infty|u^*_{\eta,\nu}(s)|^ps^{\alpha}\mathrm ds.
    \end{align*}
\end{proof}

Next, we present another technical lemma.
\begin{lemma}\label{lemma22}
Let $p>1$, $\beta>-\alpha_1/(p-1)$, and $\gamma>\beta-1+(\alpha_1-\beta)/p$. If $u\in X^{1,p}_{R}(\alpha_0,\alpha_1)$, then
\begin{equation*}
1\leq\dfrac{(-\mu_{\beta,u}'(t))^{p-1}}{[(\gamma+1)\mu_{\gamma,u}(t)]^{\frac{p\beta+\alpha_1-\beta}{\gamma+1}}}\left(-\dfrac{\mathrm d}{\mathrm dt}\int_{\{|u|>t\}}|u'|^pr^{\alpha_1}\mathrm dr\right).
\end{equation*}
\end{lemma}
\begin{proof}
Set $\theta=\beta+(\alpha_1-\beta)/p$. For fixed $t,h>0$, applying H\"older's inequality, we get
\begin{equation*}
\dfrac1h\int_{\{t<|u|\leq t+h\}}|u'|r^\theta\mathrm dr\leq\left(\dfrac1h\int_{\{t<|u|\leq t+h\}}|u'|^pr^{\alpha_1}\right)^{\frac1p}\left(\dfrac{\mu_{\beta,u}(t)-\mu_{\beta,u}(t+h)}{h}\right)^{\frac{p-1}p}.
\end{equation*}
By letting $h\downarrow0^+$, we obtain
\begin{equation}\label{eqra1}
-\dfrac{\mathrm d}{\mathrm dt}\int_{\{t<|u|\}}|u'|r^\theta\mathrm dr\leq\left(-\dfrac{\mathrm d}{\mathrm dt}\int_{\{t<|u|\}}|u'|^pr^{\alpha_1}\mathrm dr\right)^{\frac1p}\left(-\mu_{\beta,u}'(t)\right)^{\frac{p-1}p}.
\end{equation}
On the other hand, from the coarea formula and the isoperimetric inequality,
\begin{equation*}
\int_{\{t<|u|\}}|u'|r^{\theta}\mathrm dr=\int_t^\infty\int_{u^{-1}(s)}r^\theta dH_0(r)\mathrm ds+\int_{-\infty}^{-t}\int_{u^{-1}(s)}r^\theta \mathrm dH_0(r)\mathrm ds.
\end{equation*}
Let $\Omega_t=\{t<|u|\}$, then $\partial\Omega_t=u^{-1}(t)\cup u^{-1}(-t)$. Thus
\begin{equation}\label{eqra2}
-\dfrac{\mathrm d}{\mathrm dt}\int_{\{t<|u|\}}|u'|r^\theta\mathrm dr=\int_{\partial\Omega_t}r^\theta\mathrm dH_0(r).
\end{equation}
Since $\gamma+1>\theta$, we can apply the Isoperimetric Inequality on $\mathbb R$ with weight (see Theorem 6.1 in \cite{MR3619238}) to obtain
\begin{equation}\label{eqra3}
\int_{\partial\Omega}|r|^\theta\mathrm dH_0(r)\geq \left((\gamma+1)\int_\Omega|r|^\gamma\mathrm dr\right)^{\frac{\theta}{\gamma+1}}
\end{equation}
for any $\Omega\subset\mathbb R$. By using \eqref{eqra1}, \eqref{eqra2}, and \eqref{eqra3}, we conclude the lemma.
\end{proof}

Let $f\in L^p_\sigma$ with $p>1$. We consider the following problem with $0<R<\infty$:
\begin{equation}\label{eq7}
    \left\{\begin{array}{l}
         L_{\eta,\gamma} u=f\mbox{ in }(0,R),  \\
         u'(0)=u(R)=0.
    \end{array}\right.
\end{equation}
We say that $u\in X^{1,2}_{0,R}(\alpha_0,\gamma)$ is a weak solution of \eqref{eq7} if
\begin{equation*}
\int_0^Ru'\varphi'r^{\gamma}\mathrm dr=\int_0^Rf\varphi r^{\eta}\mathrm dr,\quad\forall \varphi\in A,
\end{equation*}
where $A=\{\varphi\in L^{\frac{p}{p-1}}_{\frac{p\eta-\sigma}{p-1}}\colon \varphi'\in L^2_{\gamma}\}$.
\begin{lemma}\label{lemma23}
Let $u\in X^{1,2}_{0,R}(\alpha_0,\gamma)$ be a weak solution to \eqref{eq7}. If $\eta>(\sigma-p+1)/p$, then
\begin{equation*}
-\dfrac{\mathrm d}{\mathrm{dt}}\int_{\{|u|>t\}}|u'|^2r^{{\gamma}}\mathrm dr\leq \int_0^{[(\nu+1)\mu_{\eta,u}(t)]^{\frac1{\nu+1}}}f^*_{\eta,\nu}(r)r^{\nu}\mathrm dr\quad\mbox{a.e. }t>0,
\end{equation*}
where $f^*_{\eta,\nu}$ is the half $\mu_{\eta,\nu}$-symmetrization of $f$.
\end{lemma}
\begin{proof}
    For fixed $t,h>0$, we define
    \begin{equation*}
    \varphi(r):=\left\{\begin{array}{ll}
         0&\mbox{if }|u|\leq t,  \\
         (|u|-t)\ \mathrm{sign}(u)&\mbox{if }t\leq|u|\leq t+h,\\
         h\ \mathrm{sign}(u)&\mbox{if }t+h<|u|.
    \end{array}\right.
    \end{equation*}
Since $\eta>(\sigma-p+1)/p$, we have $\varphi\in A$. By considering that $u$ is a weak solution, we obtain
\begin{equation*}
\int_0^Ru'\varphi'r^{{\gamma}}\mathrm dr=\int_0^Rf\varphi r^\eta\mathrm dr.
\end{equation*}
Thus,
\begin{align}
    \underset{\{t<|u|\leq t+h\}}{\int}&|u'|^2r^{{\gamma}}\mathrm dr\nonumber\\
    &=\underset{\{t<|u|\leq t+h\}}{\int}f(|u|-t)\ \mathrm{sign}(u)r^{\eta}\mathrm dr+\underset{\{|u|>t+h\}}{\int}fh\ \mathrm{sign}(u)r^{\eta}\mathrm dr\nonumber\\
    &\leq\underset{\{t<|u|\leq t+h\}}{\int}|f||u|r^{\eta}\mathrm dr-t\underset{\{t<|u|\leq t+h\}}{\int}|f|r^{\eta}\mathrm dr+h\underset{\{|u|>t+h\}}{\int}|f|r^{\eta}\mathrm dr.\label{eq8}
\end{align}
Dividing \eqref{eq8} by $h$ and letting $h\downarrow0^+$, we obtain
\begin{equation*}
-\dfrac{\mathrm d}{\mathrm{d}t}\underset{\{|u|>t\}}{\int}|u'|^2r^{\gamma}\mathrm dr\leq-\dfrac{\mathrm d}{\mathrm dt}\underset{\{|u|>t\}}{\int}|f|(|u|-t)r^\eta\mathrm dr.
\end{equation*}
By Leibniz integral rule and Corollary \ref{cor21},
\begin{equation*}
-\dfrac{\mathrm d}{\mathrm{d}t}\underset{\{|u|>t\}}{\int}|u'|^2r^{\gamma}\mathrm dr\leq\underset{\{|u|>t\}}{\int}|f|r^{\eta}\mathrm dr=\int_0^{[(\nu+1)\mu_{\eta,u}(t)]^{\frac1{\nu+1}}}f^*_{\eta,\nu}(r)r^{\nu}\mathrm dr.
\end{equation*}
\end{proof}

The next proposition is an adaptation of \cite[Proposition 3.4]{MR3225631} to the half $\mu_{\eta,\nu}$-symmetrization.

\begin{prop}\label{prop22}
 Let $u\in X^{1,2}_{0,R}(\alpha_0,\gamma)$ be a weak solution of \eqref{eq7} with $f\in L^p_\sigma$. If $\eta>(\sigma-p+1)/p$ and $1<\gamma<\eta+2$, then
\begin{equation*}
u^*_{\eta,\nu}(r_1)-u^*_{\eta,\nu}(r_2)\leq\dfrac{(\nu+1)^{\frac{\gamma-1}{\eta+1}}}{(\eta+1)^{\frac{\eta+\gamma}{\eta+1}}}\int_{r_1}^{r_2}f^{**}_{\eta,\nu}(\xi)\xi^{\nu-\frac{\gamma-1}{\eta+1}(\nu+1)}\mathrm d\xi,\quad\forall 0<r_1\leq r_2\leq R,
\end{equation*}
where $u^*_{\eta,\nu}$ is the half $\mu_{\eta,\nu}$-symmetrization and $f^{**}_{\eta,\nu}$ is given by \eqref{plb}.
\end{prop}
\begin{proof}
From $\eta>(\nu-p+1)/p$ and $1<\gamma<\eta+2$ we can use Lemmas \ref{lemma22} and \ref{lemma23} to obtain
\begin{equation*}
1\leq\dfrac{-\mu_{\eta,u}'(t)}{[(\eta+1)\mu_{\eta,u}(t)]^{\frac{\eta+\gamma}{\eta+1}}}\int_0^{[(\nu+1)\mu_{\eta,u}(t)]^{\frac{1}{\nu+1}}}f^*_{\eta,\nu}(r)r^{\nu}\mathrm dr.
\end{equation*}
Considering $0<s'<s$ and integrating on $t$ from $s'$ to $s$, we have
\begin{equation*}
s-s'\leq\int_s^{s'}\left[(\eta+1)\mu_{\eta,u}(t)\right]^{-\frac{\eta+\gamma}{\eta+1}}\int_0^{[(\nu+1)\mu_{\eta,u}(t)]^{\frac{1}{\nu+1}}}f^*_{\eta,\nu}(r)r^{\nu}\mathrm dr\mu'_{\eta,u}(t)\mathrm dt.
\end{equation*}
Making the variable changing $\xi=[(\nu+1)\mu_{\eta,u}(t)]^{\frac{1}{\nu+1}}$ and using $\int_0^\xi f^*_{\eta,\nu}(r)r^{\nu}\mathrm dr=\xi^{\nu+1}f^{**}_{\eta,\nu}(\xi)/(\nu+1)$, we get
\begin{equation}\label{eq10}
    s-s'\leq\dfrac{(\nu+1)^{\frac{\gamma-1}{\eta+1}}}{(\eta+1)^{\frac{\eta+\gamma}{\eta+1}}}\int_{[(\nu+1)\mu_{\eta,u}(s)]^{\frac{1}{\nu+1}}}^{[(\nu+1)\mu_{\eta,u}(s')]^{\frac{1}{\nu+1}}}f^{**}_{\eta,\nu}(\xi)\xi^{\nu-\frac{\gamma-1}{\eta+1}(\nu+1)}\mathrm d\xi.
\end{equation}

Let $0<r_1<r_2\leq R$. Without loss of generality, we may assume $u^*_{\eta,\nu}(r_2)<u^*_{\eta,\nu}(r_1)$. Then there exists $\overline\delta>0$ such that
\begin{equation*}
u^*_{\eta,\nu}(r_2)+\delta<u^*_{\eta,\nu}(r_1),\quad\forall\delta\in[0,\overline\delta].
\end{equation*}
Fixing $\delta\in(0,\overline \delta]$, we apply equation \eqref{eq10} for
\begin{equation*}
s=u^*_{\eta,\nu}(r_1)-\dfrac\delta2\mbox{ and }s'=u^*_{\eta,\nu}(r_2)+\dfrac{\delta}2
\end{equation*}
to obtain
\begin{equation*}
u^*_{\eta,\nu}(r_1)-u^*_{\eta,\nu}(r_2)-\delta\leq\dfrac{(\nu+1)^{\frac{\gamma-1}{\eta+1}}}{(\eta+1)^{\frac{\eta+\gamma}{\eta+1}}}\int_{[(\nu+1)\mu_{\eta,u}(s)]^{\frac{1}{\nu+1}}}^{[(\nu+1)\mu_{\eta,u}(s')]^{\frac{1}{\nu+1}}}f^{**}_{\eta,\nu}(\xi)\xi^{\nu-\frac{\gamma-1}{\eta+1}(\nu+1)}\mathrm d\xi.
\end{equation*}
By $u^*_{\eta,\nu}(r_1)>s$ and $u^*_{\eta,\nu}(r_2)<s'$, we get $\mu_{\eta,u}(s)\geq r_1^{\nu+1}/(\nu+1)$ and
$\mu_{\eta,u}(s')< r_2^{\nu+1}/(\nu+1)$. Then
\begin{equation*}
u^*_{\eta,\nu}(r_1)-u^*_{\eta,\nu}(r_2)-\delta\leq\dfrac{(\nu+1)^{\frac{\gamma-1}{\eta+1}}}{(\eta+1)^{\frac{\eta+\gamma}{\eta+1}}}\int_{r_1}^{r_2}f^{**}_{\eta,\nu}(\xi)\xi^{\nu-\frac{\gamma-1}{\eta+1}(\nu+1)}\mathrm d\xi.
\end{equation*}
Letting $\delta\to0$, we conclude the result.
\end{proof}

\subsection{Crucial Lemma}

In this subsection, we establish a crucial lemma (Lemma \ref{lemmakey}) that enables us to derive Adams' inequality with exact growth for the second derivative case. However, before presenting the lemma, we need to recall a technical result (Lemma \ref{lemmakey2}):

\begin{lemma}[Lemma 3.4 in \cite{MR3336837}]\label{lemmakey2}
For any given sequence $a=(a_k)_{k\geq0}$, let $p>1$, $\|a\|_1=\sum_{k=0}^\infty|a_k|$, $\|a\|_p=(\sum_{k=0}^\infty|a_k|^p)^{1/p}$, $\|a\|_{(e)}=(\sum_{k=0}^{\infty}|a_k|^pe^k)^{1/p}$, and $\mu(h)=\inf\{\|a\|_{(e)}\colon\|a\|_1=h,\|a\|_p\leq 1\}$. Then, for $h>1$, there exist positive constants $C_1$ and $C_2$ such that
\begin{equation*}
C_1\mu(h)\leq\dfrac{e^{\frac{h^{\frac{p}{p-1}}}{p}}}{h^{\frac{1}{p-1}}}\leq C_2\mu(h).
\end{equation*}
\end{lemma}

Now, we can state the main result in this subsection:

\begin{lemma}\label{lemmakey}
Assume $\alpha_2=2p-1>1$, $\alpha_0\geq-1$, $\eta>-1$, $\theta\in\mathbb R$, and $\gamma=(\alpha_2+(p-1)\eta)/p$. Let $u\in X^{2,p}_\infty(\alpha_0,\alpha_1,\alpha_2)$ with compact support in $[0,\infty)$. If $u^*_\eta(R)\geq1$ and $g:=r^{\theta-\eta}L_{\theta,\gamma} u$ satisfies $g\in L^p_{\alpha_2+(\eta-\gamma)p}$ and
\begin{equation*}
\int_{ R}^\infty|g^{**}_\eta|^pr^\eta\mathrm dr\leq \left(\dfrac{p}{p-1}\right)^p,
\end{equation*}
then
\begin{equation*}
\dfrac{\exp(\beta_{0,2}|u^*_\eta(R)|^{\frac{p}{p-1}})}{|u^*_\eta(R)|^{\frac{p}{p-1}}}R^{\eta+1}\leq C\int_R^\infty |u^*_\eta(r)|^pr^\eta\mathrm dr,
\end{equation*}
where $C=C(p,\eta)>0$, $g^{**}_\eta$ is given by \eqref{plb} with $\nu=\eta$, and $u^*_\eta$ represents the half $\mu_{\eta,\eta}$-symmetrization.
 \end{lemma}
 \begin{proof}
Let $\overline R>0$ such that $\mathrm{supp}(u)\subset[0,\overline R]$. By Theorem \ref{theo32}, we obtain that $u\in X^{1,2}_{0,\overline R}(\alpha_0,\gamma)$ and is a weak solution of
\begin{equation*}
    \left\{\begin{array}{l}
         L_{\eta,\gamma} u=g\mbox{ in }(0, \overline R),  \\
         u'(0)=u( \overline R)=0.
    \end{array}\right.
\end{equation*}
Using Proposition \ref{prop22} on this problem with $\gamma=(\alpha_2+(p-1)\eta)/p$, we have
\begin{equation}\label{eql1}
u^*_\eta(r_1)-u^*_\eta(r_2)\leq\dfrac{1}{\eta+1}\int_{ r_1}^{ r_2}g^{**}_\eta(r)r^{\frac{\eta-p+1}{p}}\mathrm dr,\quad\forall 0<r_1\leq r_2\leq \overline R.
\end{equation}

Let $h_k=c_0u^*_\eta(Re^{\frac{k}{\eta+1}})$ for each $k\in\mathbb N\cup\{0\}$, where $c_0=(\eta+1)^{\frac{2p-1}p}(p-1)/p$. Define $a_k=h_k-h_{k+1}$ and $a=(a_k)$. Then $a_k\geq0$ and
\begin{equation*}
        \|a\|_1=\sum_{k=0}^\infty|a_k|=h_0=c_0u^*_\eta(R).
\end{equation*}
By \eqref{eql1} and H\"older's inequality,
\begin{align*}
a_k&=c_0|u^*_\eta(Re^{\frac{k}{\eta+1}})-u^*_\eta(Re^{\frac{k+1}{\eta+1}})|\\
   &\leq \dfrac{c_0}{\eta+1}\left(\int_{ Re^{\frac{k}{\eta+1}}}^{ Re^{\frac{k+1}{\eta+1}}}|g^{**}_\eta(r)|^pr^\eta\mathrm dr\right)^{\frac{1}{p}}(\eta+1)^{\frac{1-p}p}\\
   &\leq\dfrac{p-1}p\left(\int_{Re^{\frac{k}{\eta+1}}}^{Re^{\frac{k+1}{\eta+1}}}|g^{**}_\eta(r)|^pr^\eta\mathrm dr\right)^{\frac{1}{p}}.
\end{align*}
Then
\begin{equation*}
\|a\|_p=\left(\sum_{k=0}^\infty|a_k|^p\right)^{\frac1p}\leq\left[\left(\dfrac{p-1}p\right)^{p} \int_R^\infty|g^{**}_\eta(r)|^pr^\eta\mathrm dr\right]^{\frac1p}\leq 1.
\end{equation*}
On the other hand,
\begin{align*}
R^{-\eta-1}\int_R^\infty|u^*_\eta(r)|^pr^\eta\mathrm dr&\geq R^{-\eta-1}\sum_{k=0}^\infty\int_{Re^{\frac{k}{\eta+1}}}^{Re^{\frac{k+1}{\eta+1}}}|u^*_\eta(Re^{\frac{k+1}{\eta+1}})|^pr^\eta\mathrm dr\\
&=\dfrac{1-e^{-1}}{\eta+1}\sum_{k=0}^\infty|u^*_\eta(Re^{\frac{k+1}{\eta+1}})|^pe^{k+1}\\
&\geq C\sum_{k=0}^\infty |h_{k+1}|^pe^{k+1}\geq C\sum_{k=1}^\infty |a_{k}|^pe^k.
\end{align*}
Thus,
\begin{equation}\label{eqpt10}
\|a\|_{(e)}^p=a_0^p+\sum_{k=1}^\infty |a_k|^pe^k\leq h_0^p+CR^{-\eta-1}\int_R^\infty|u^*_\eta(r)|^pr^\eta\mathrm dr.
\end{equation}
Let $R<r<Re^b$, where $b=[(\eta+1)(p-1)/2p]^{\frac{p}{p-1}}$. Note that, by \eqref{eql1},
\begin{align*}
h_0-c_0u^*_\eta(r)&\leq\dfrac{c_0}{\eta+1}\int_{ R}^{ r}g^{**}_\eta(s)s^{\frac{\eta-p+1}{p}}\mathrm ds\leq\dfrac{c_0}{\eta+1}\left(\int_{ R}^{ r}|g^{**}_\eta(s)|^ps^\eta\mathrm ds\right)^{\frac1p}b^{\frac{p-1}p}\\
&\leq \dfrac{b^{\frac{p-1}p}c_0}{\eta+1}\dfrac{p}{p-1}\leq\dfrac{h_0}{2}.
\end{align*}
Then $h_0\leq Cu^*_\eta(r)$ for all $R<r<Re^b$. Using
\begin{equation*}
    h_0^p\leq C\dfrac{\int_R^{Re^b}h_0^pr^\eta\mathrm dr}{R^{\eta+1}}\leq C\dfrac{\int_R^{\infty}|u^*_\eta(r)|^pr^\eta\mathrm dr}{R^{\eta+1}}
\end{equation*}
along with \eqref{eqpt10}, we have
\begin{equation*}
\|a\|_{(e)}^p\leq C\dfrac{\int_R^{\infty}|u^*_\eta(r)|^pr^\eta\mathrm dr}{R^{\eta+1}}.
\end{equation*}
From Lemma \ref{lemmakey2} and $c_0^{\frac{p}{p-1}}=\beta_{0,2}$, we conclude
\begin{equation*}
C\dfrac{\int_R^{\infty}|u^*_\eta(r)|^pr^\eta\mathrm dr}{R^{\eta+1}}\geq\left[\frac{\exp\left(\frac{h_0^{\frac{p}{p-1}}}{p}\right)}{h_0^{\frac{1}{p-1}}}\right]^p=\dfrac{\exp\left(\beta_{0,2}|u^*_\eta(R)|^{\frac{p}{p-1}}\right)}{c_0^{\frac{p}{p-1}}|u^*_\eta(R)|^{\frac{p}{p-1}}}.
\end{equation*}
Therefore,
\begin{equation*}
\dfrac{\exp(\beta_{0,2}|u^*_\eta(R)|^{\frac{p}{p-1}})}{|u^*_\eta(R)|^{\frac{p}{p-1}}}R^{\eta+1}\leq C\int_R^\infty |u^*_\eta(r)|^pr^\eta\mathrm dr,
\end{equation*}
which completes the proof of Lemma \ref{lemmakey}.
 \end{proof}

\section{Proof of the Exact Growth for First and Second Order}\label{mainresults}

As mentioned in \cite{MR3405815,MR3848068,MR3225631,MR3355498}, establishing the exact growth inequality requires a crucial lemma like Lemma \ref{lemmakey}. While the proof of both cases shares similarities, we provide a comprehensive proof specifically for the second derivative case due to its additional intricacies. Once we have proven the second order result (Theorem \ref{theo1}), we will address the approach for handling the first derivative case.

\begin{proof}[Proof of Theorem \ref{theo1}]
From Lemma \ref{lemmacompsuppdense} and $\|L_{\theta,\gamma} u\|_{L^p_{\alpha_2+(\theta-\gamma)p}}\leq C\|u\|_{X^{2,p}_\infty}$ (using $\alpha_1\geq p-1$ and Lemma \ref{lemmajaosn}) we can assume that $\mathrm{supp}(u)\subset[0,\infty)
$ is compact and $u\in C^\infty(0,\infty)$. From Corollary \ref{cor21}, it is enough to show that
\begin{equation*}
\int_0^\infty\dfrac{\exp_p(\beta_{0,2}|u^*_{\eta}|^{\frac{p}{p-1}})}{(1+|u^*_{\eta}|)^{\frac{p}{p-1}}}r^\eta\mathrm dr\leq C\|u^*_{\eta}\|^{p}_{L^{p
}_{\eta}},
\end{equation*}
where $u^*_{\eta}$ is the half $\mu_{\eta,\eta}$-symmetrization of $u$. We split the integral into two parts
\begin{equation*}
\int_0^\infty\dfrac{\exp_p(\beta_{0,2}|u^*_{\eta}|^{\frac{p}{p-1}})}{(1+|u^*_{\eta}|)^{\frac{p}{p-1}}}r^\eta\mathrm dr=\int_0^{R_0}\dfrac{\exp_p(\beta_{0,2}|u^*_{\eta}|^{\frac{p}{p-1}})}{(1+|u^*_{\eta}|)^{\frac{p}{p-1}}}r^\eta\mathrm dr+\int_{R_0}^\infty\dfrac{\exp_p(\beta_{0,2}|u^*_{\eta}|^{\frac{p}{p-1}})}{(1+|u^*_{\eta}|)^{\frac{p}{p-1}}}r^\eta\mathrm dr,
\end{equation*}
where $R_0=\inf\{r\geq0\colon u^*_{\eta}(r)\leq1\}\in[0,\infty)$. Note that $u^*_{\eta}(r)\leq1$ for all $r\in(R_0,\infty)$ and then $\exp_p(\beta_{0,2}|u^*_{\eta}|^{\frac{p}{p-1}})\leq C|u^*_{\eta}|^p$ in $r\in(R_0,\infty)$. Thus,
\begin{equation*}
\int_{R_0}^\infty\dfrac{\exp_p(\beta_{0,2}|u^*_{\eta}|^{\frac{p}{p-1}})}{(1+|u^*_{\eta}|)^{\frac{p}{p-1}}}r^\eta\mathrm dr\leq C\|u^*_{\eta}\|^p_{L^p_\eta}.
\end{equation*}
Therefore, we just need to consider the integral on $(0,R_0)$. 

Let $R_1>0$ such that $\mathrm{supp}(u)\subset[0,R_1]$. We define $f:=L_{\theta,\gamma} u\in L^p_{\alpha_2+(\theta-\gamma)p}$ and $g:=r^{\theta-\eta}f\in L^p_{\alpha_2+(\eta-\gamma)p}$. Using Theorem \ref{theo32}, we have $u\in X^{1,2}_{0,R_1}(\alpha_0,\gamma)$ and it is a weak solution of the following equation:
\begin{equation}\label{eqpt1}
    \left\{\begin{array}{l}
         L_{\eta,\gamma} u=g\mbox{ in }(0,R_1),  \\
         u'(0)=u(R_1)=0.
    \end{array}\right.
\end{equation}
Using $\gamma=(\alpha_2+(p-1)\eta)/p$ and Proposition \ref{prop22} on the problem \eqref{eqpt1}, we obtain
\begin{equation}\label{eq14}
u^*_\eta(r_1)-u^*_\eta(r_2)\leq\dfrac{1}{\eta+1}\int_{r_1}^{r_2}g^{**}_\eta(r)r^{\frac{\eta-p+1}{p}}\mathrm dr,\quad\forall 0<r_1\leq r_2\leq R_1,
\end{equation}
where $g^{**}_\eta$ is given by \eqref{plb} with $\nu=\eta$. Defining $\alpha:=\int_0^\infty |g^{**}_\eta|^pr^{\eta}\mathrm dr$, Lemma \ref{lemma21} and Corollary \ref{cor21} imply
\begin{equation}\label{eqalpha}
\alpha\leq \left(\dfrac{p}{p-1}\right)^p\int_0^\infty|g^*_\eta|^pr^{\eta}\mathrm dr=\left(\dfrac{p}{p-1}\right)^{p}\int_0^\infty|f|^pr^{p\theta-(p-1)\eta}\mathrm dr\leq\left(\dfrac{p}{p-1}\right)^{p},
\end{equation}
where we used that $\gamma=(\alpha_2+(p-1)\eta)/p$ and $\|f\|_{L^p_{\alpha_2+(\theta-\gamma)p}}\leq1$.

Let $0<\varepsilon_0<1$ be fixed, and define $R_2$ such that
\begin{equation}\label{eqpt3}
\int_0^{ R_2}|g^{**}_\eta|^pr^{\eta}\mathrm dr=\alpha\varepsilon_0\mbox{ and }\int_{ R_2}^\infty|g^{**}_\eta|^pr^{\eta}\mathrm dr=\alpha(1-\varepsilon_0).
\end{equation}
From \eqref{eq14}, we obtain
\begin{equation}\label{eqpt2}
u^*_\eta(r_1)-u^*_\eta(r_2)\leq\dfrac1{\eta+1}\left(\int_{ r_1}^{ r_2}|g^{**}_\eta(r)|^pr^{\eta}\mathrm dr\right)^{\frac1p}\left(\log\dfrac{r_2}{r_1}\right)^{\frac{p-1}p},
 \end{equation}
for all $0<r_1\leq r_2\leq R_1$. Using \eqref{eqpt3} and $\gamma=(\alpha_2+\eta(p-1))/p$, \eqref{eqpt2} becomes
\begin{equation}\label{eqpt4}
    u^*_\eta(r_1)-u^*_\eta(r_2)\leq\dfrac{\varepsilon_0^{\frac1p}}{\gamma-1}\left(\log \dfrac{r_2}{r_1}\right)^{\frac{p-1}p}\quad\forall 0<r_1\leq r_2\leq R_2.
\end{equation}
\begin{equation}\label{eqpt5}
    u^*_\eta(r_1)-u^*_\eta(r_2)\leq\dfrac{(1-\varepsilon_0)^{\frac1p}}{\gamma-1}\left(\log \dfrac{r_2}{r_1}\right)^{\frac{p-1}p}\quad\forall R_2\leq r_1\leq r_2.
\end{equation}

We split the proof into the cases $R_2\geq R_0$ and $R_2<R_0$. Firstly we consider the case $R_2\geq R_0$. By \eqref{eqpt4}, for all $r\in(0,R_0]$ we have
\begin{equation*}
u^*_\eta(r)\leq 1+\dfrac{\varepsilon_0^{\frac1p}}{\gamma-1}\left(\log \dfrac{R_0}{r}\right)^{\frac{p-1}p}.
\end{equation*}
Given $\varepsilon>0$, it is known that
\begin{equation}\label{eqineqepsilon}
(1+x)^{\frac{p}{p-1}}\leq (1+\varepsilon)x^{\frac{p}{p-1}}+C_\varepsilon,\quad\forall x>0,
\end{equation}
where $C_\varepsilon=(1-(1+\varepsilon)^{1-p})^{\frac{1}{1-p}}$. Thus,
\begin{equation}\label{eqpt6}
|u^*_\eta(r)|^{\frac{p}{p-1}}\leq (1+\varepsilon)\dfrac{\varepsilon_0^{\frac{1}{p-1}}}{(\gamma-1)^{\frac{p}{p-1}}}\log\dfrac{R_0}r+C_\varepsilon.
\end{equation}
Taking $\varepsilon>0$ small with $(1+\varepsilon)\varepsilon_0^{\frac1{p-1}}<1$ and using $\beta_{0,2}=(\eta+1)(\gamma-1)^{\frac{p}{p-1}}$ and \eqref{eqalpha}, we obtain
\begin{equation*}
    \beta_{0,2}(1+\varepsilon)\dfrac{\varepsilon_0^{\frac{1}{p-1}}}{(\gamma-1)^{\frac{p}{p-1}}}<\eta+1.
\end{equation*}
Then, by \eqref{eqpt6} and $\exp_p(t)\leq e^t$ (see \cite[Lemma 3.1]{arXiv:2306.00194}),
\begin{align*}
\int_0^{R_0}\dfrac{\exp_p(\beta_{0,2}|u^*_\eta|^{\frac{p}{p-1}})}{(1+|u^*_\eta|)^{\frac{p}{p-1}}}r^\eta\mathrm dr&\leq e^{C_\varepsilon}\int_0^{R_0}\exp\left(\beta_{0,2}(1+\varepsilon)\dfrac{\varepsilon_0^{\frac{1}{p-1}}}{(\gamma-1)^{\frac{p}{p-1}}}\log\dfrac{R_0}r\right)r^\eta\mathrm dr\\
&\leq e^{C_\varepsilon}R_0^{\beta_{0,2}(1+\varepsilon)\frac{\varepsilon_0^{\frac{1}{p-1}}}{(\gamma-1)^{\frac{p}{p-1}}}}\int_0^{R_0}r^{\eta-\beta_{0,2}(1+\varepsilon)\frac{\varepsilon_0^{\frac{1}{p-1}}}{(\gamma-1)^{\frac{p}{p-1}}}}\mathrm dr\\
&=CR_0^{\eta+1}\leq C\|u^*_\eta\|_{L^p_\eta}^p.
\end{align*}
This concludes the case $R_2\geq R_0$.

For the case $R_2<R_0$, we observe that:
\begin{equation*}
\int_0^{R_0}\dfrac{\exp_p(\beta_{0,2}|u^*_\eta|^{\frac{p}{p-1}})}{(1+|u^*_\eta|)^{\frac{p}{p-1}}}r^\eta\mathrm dr=\int_{R_2}^{R_0}\dfrac{\exp_p(\beta_{0,2}|u^*_\eta|^{\frac{p}{p-1}})}{(1+|u^*_\eta|)^{\frac{p}{p-1}}}r^\eta\mathrm dr+\int_0^{R_2}\dfrac{\exp_p(\beta_{0,2}|u^*_\eta|^{\frac{p}{p-1}})}{(1+|u^*_\eta|)^{\frac{p}{p-1}}}r^\eta\mathrm dr.
\end{equation*}
To estimate the integral over $(R_2,R_0)$, we use \eqref{eqpt5} to obtain
\begin{equation*}
u^*_\eta(r)\leq 1+\dfrac{(1-\varepsilon_0)^{\frac1p}}{\gamma-1}\left(\log\dfrac{R_0}r\right)^{\frac{p-1}p},\quad\forall r\in(R_2,R_0).
\end{equation*}
By setting $\varepsilon_1>0$ small such that $(1+\varepsilon_1)(1-\varepsilon_0)^{\frac{1}{p}}<1$, \eqref{eqineqepsilon}
 implies:
\begin{equation*}
|u^*_\eta(r)|^{\frac{p}{p-1}}\leq(1+\varepsilon_1)\dfrac{(1-\varepsilon_0)^{\frac{1}{p-1}}}{(\gamma-1)^{\frac{p}{p-1}}}\log\dfrac{R_0}r+C_{\varepsilon_1}.
\end{equation*}
Hence, using the expression $\beta_{0,2}=(\eta+1)(\gamma-1)^{\frac{p}{p-1}}$ and $\exp_p(t)\leq e^t$ (refer to \cite[Lemma 3.1]{arXiv:2306.00194}), we can derive the following inequality:
\begin{align*}
\int_{R_2}^{R_0}\dfrac{\exp_p(\beta_{0,2}|u^*_\eta|^{\frac{p}{p-1}})}{(1+|u^*_\eta|)^{\frac{p}{p-1}}}r^\eta\mathrm dr&\leq e^{C_{\varepsilon_1}}\int_{R_2}^{R_0}\exp\left(\beta_{0,2}(1+\varepsilon_1)\dfrac{(1-\varepsilon_0)^{\frac{1}{p-1}}}{(\gamma-1)^{\frac{p}{p-1}}}\log\dfrac{R_0}r\right)r^\eta\mathrm dr\\
&\leq e^{C_{\varepsilon_1}}R_0^{\beta_{0,2}(1+\varepsilon_1)\frac{(1-\varepsilon_0)^{\frac{1}{p-1}}}{(\gamma-1)^{\frac{p}{p-1}}}}\int_0^{R_0}r^{\eta-\beta_{0,2}(1+\varepsilon_1)\frac{(1-\varepsilon_0)^{\frac{1}{p-1}}}{(\gamma-1)^{\frac{p}{p-1}}}}\mathrm dr\\
&=CR_0^{\eta+1}\leq C\|u^*_\eta\|_{L^p_\eta}^p.
\end{align*}
Consequently, the integral over the interval $(0,R_2)$ remains to be considered. Notably, according to \eqref{eqineqepsilon}, we have
\begin{equation*}
|u^*_\eta(r)|^{\frac{p}{p-1}}\leq(1+\varepsilon_2)|u^*_\eta(r)-u^*_\eta(R_2)|^{\frac{p}{p-1}}+C_{\varepsilon_2}|u^*_\eta(R_2)|^{\frac{p}{p-1}},
\end{equation*}
for all $\varepsilon_2>0$ and $0<r<R_2$. Hence, we obtain
\begin{align*}
\int_{0}^{R_2}\dfrac{\exp_p(\beta_{0,2}|u^*_\eta|^{\frac{p}{p-1}})}{(1+|u^*_\eta|)^{\frac{p}{p-1}}}r^\eta\mathrm dr&\leq \dfrac1{|u^*_\eta(R_2)|^{\frac{p}{p-1}}}\int_0^{R_2}\exp(\beta_{0,2}|u^*_\eta|^{\frac{p}{p-1}})r^\eta\mathrm dr\\
&\leq\dfrac{\exp(C_{\varepsilon_2}\beta_{0,2}|u^*_\eta(R_2)|^{\frac{p}{p-1}})}{|u^*_\eta(R_2)|^{\frac{p}{p-1}}}\\
&\quad\times\int_0^{R_2}\exp(\beta_{0,2}(1+\varepsilon_2)|u^*_\eta(r)-u^*_\eta(R_2)|^{\frac{p}{p-1}})r^\eta\mathrm dr.
\end{align*}
Using the expression \eqref{eq14}, we have
\begin{align}
\int_{0}^{R_2}&\dfrac{\exp_p(\beta_{0,2}|u^*_\eta|^{\frac{p}{p-1}})}{(1+|u^*_\eta|)^{\frac{p}{p-1}}}r^\eta\mathrm dr\leq\dfrac{\exp(C_{\varepsilon_2}\beta_{0,2}|u^*_\eta(R_2)|^{\frac{p}{p-1}})}{|u^*_\eta(R_2)|^{\frac{p}{p-1}}}\nonumber\\
&\quad\times\int_0^{R_2}\exp\left[\dfrac{\beta_{0,2}^{\frac{p-1}{p}}(1+\varepsilon_2)^{\frac{p-1}p}}{\eta+1}\int_{ r}^{ R_2}g^{**}_\eta(s)s^{\frac{\eta-p+1}p}\mathrm ds\right]^{\frac{p}{p-1}}r^\eta\mathrm dr\nonumber\\
&=R_2^{\eta+1}\dfrac{\exp(C_{\varepsilon_2}\beta_{0,2}|u^*_\eta(R_2)|^{\frac{p}{p-1}})}{(\eta+1)|u^*_\eta(R_2)|^{\frac{p}{p-1}}}\nonumber\\
&\quad\times\int_0^\infty\exp\left[\dfrac{\beta_{0,2}^{\frac{p-1}p}(1+\varepsilon_2)^{\frac{p-1}p}}{\eta+1}\int_{ R_2e^{-\frac{t}{\eta+1}}}^{ R_2}g^{**}_\eta(s)s^{\frac{\eta-p+1}p}\mathrm ds\right]^{\frac{p}{p-1}}e^{-t}\mathrm dt\label{eqpt7},
\end{align}
where the change of variable $r=R_2e^{-\frac{t}{\eta+1}}$ has been applied.

Let $\phi\colon\mathbb R\to \mathbb R$ be defined as $\phi(t)=0$ for $t\in(-\infty,0]$, and for $t\in(0,\infty)$, we define
\begin{equation*}
\phi(t)=\dfrac{\beta_{0,2}^{\frac{p-1}p}(1+\varepsilon_2)^{\frac{p-1}p}R_2^{\frac{\eta+1}{p}}}{(\eta+1)^2}g^{**}_\eta(R_2e^{-\frac{t}{\eta+1}})e^{-\frac{t}{p}}.
\end{equation*}
By choosing $\varepsilon_2=\varepsilon_0^{\frac1{1-p}}-1$, we can apply \eqref{eqalpha}, \eqref{eqpt3}, and $\gamma=(\alpha_2+(p-1)\eta)/p$ to obtain
\begin{align}
\int_{-\infty}^\infty|\phi(t)|^p\mathrm dt&=\dfrac{\beta_{0,2}^{p-1}(1+\varepsilon_2)^{p-1}R_2^{\eta+1}}{(\eta+1)^{2p}}\int_0^{\infty}|g^{**}_\eta( R_2e^{-\frac{t}{\eta+1}})|^pe^{-t}\mathrm dt\nonumber\\
&=\dfrac{\beta_{0,2}^{p-1}(1+\varepsilon_2)^{p-1}}{(\eta+1)^{2p-1}}\int_0^{ R_2}|g^{**}_\eta(r)|^pr^{\eta}\mathrm dr\nonumber\\
&\leq (1+\varepsilon_2)^{p-1}\varepsilon_0\dfrac{\beta_{0,2}^{p-1}}{(\eta+1)^{2p-1}}\left(\dfrac{p}{p-1}\right)^p=1.\label{eqpt8}
\end{align}
Consider the following lemma due to D. R. Adams \cite[Lemma 1]{MR0960950}.
\begin{lemma}
Let $a(s,t)$ be a nonnegative measurable function on $(-\infty,\infty)\times[0,\infty)$ such that (a.e.)
\begin{equation*}
a(s,t)\leq1,\mbox{ when }0<s<t,
\end{equation*}
\begin{equation*}
\sup_{t>0}\left(\int_{-\infty}^0+\int_t^\infty a(s,t)^{\frac{p}{p-1}}\mathrm ds\right)^{\frac{p-1}p}=b<\infty.
\end{equation*}
Then there is a constant $c_0=c_0(p,b)$ such that if $\phi\geq0$,
\begin{equation*}
\int_{-\infty}^{\infty}\phi(s)^p\mathrm ds\leq1,
\end{equation*}
then
\begin{equation*}
\int_0^\infty e^{-F(t)}\mathrm dt\leq c_0,
\end{equation*}
where
\begin{equation*}
    F(t)=t-\left(\int_{-\infty}^\infty a(s,t)\phi(s)\mathrm ds\right)^{\frac{p}{p-1}}.
\end{equation*}
\end{lemma}
From \eqref{eqpt8}, we can apply the above lemma to the function $a(s,t)=\chi_{(0,t)}(s)$ ($\chi$ denotes the characteristic function) to obtain
\begin{equation*}
    \int_0^\infty\exp\left(\int_0^t\phi(s)\mathrm ds\right)^{\frac{p}{p-1}}e^{-t}\mathrm dt\leq c_0.
\end{equation*}
Using the change of variable $r= R_2e^{-\frac{s}{\eta+1}}$, we have
\begin{align*}
    c_0&\geq \int_0^\infty\exp\left(\int_0^t\phi(s)\mathrm ds\right)^{\frac{p}{p-1}}e^{-t}\mathrm dt\\
    &=\int_0^\infty\exp\left(\dfrac{\beta_{0,2}^{\frac{p-1}p}(1+\varepsilon_2)^{\frac{p-1}p}( R_2)^{\frac{\eta+1}{p}}}{(\eta+1)^2}\int_{0}^{t}g^{**}_\eta( R_2e^{-\frac{s}{\eta+1}})e^{-\frac{s}{p}}\mathrm ds\right)^{\frac{p}{p-1}}e^{-t}\mathrm dt\\
    &=\int_0^\infty\exp\left(\dfrac{\beta_{0,2}^{\frac{p-1}p}(1+\varepsilon_2)^{\frac{p-1}p}}{\eta+1}\int_{ R_2 e^{-\frac{t}{\eta+1}}}^{ R_2}g^{**}_\eta(r)r^{\frac{\eta-p+1}p}\mathrm dr\right)^{\frac{p}{p-1}}e^{-t}\mathrm dt.
\end{align*}
Here, $\varepsilon_2=\varepsilon_0^{\frac{1}{1-p}}-1$ implies $C_{\varepsilon_2}=(1-\varepsilon_0)^{\frac{1}{1-p}}$. Now, using \eqref{eqpt7}, we have
\begin{align*}
\int_{0}^{R_2}\dfrac{\exp_p(\beta_{0,2}|u^*_\eta|^{\frac{p}{p-1}})}{(1+|u^*_\eta|)^{\frac{p}{p-1}}}r^\eta\mathrm dr&\leq c_0R_2^{\eta+1}\dfrac{\exp((1-\varepsilon_0)^{\frac{1}{1-p}}\beta_{0,2}|u^*_\eta(R_2)|^{\frac{p}{p-1}})}{(\eta+1)|u^*_\eta(R_2)|^{\frac{p}{p-1}}}\\
&=CR_2^{\eta+1}\dfrac{\exp((1-\varepsilon_0)^{\frac{1}{1-p}}\beta_{0,2}|u^*_\eta(R_2)|^{\frac{p}{p-1}})}{|u^*_\eta(R_2)|^{\frac{p}{p-1}}}.
\end{align*}
Since $\int_{R_2}^\infty|g^{**}_\eta|^pr^{\eta}\mathrm dr\leq(\frac{p}{p-1})^p(1-\varepsilon_0)$, we can apply Lemma \ref{lemmakey} to $u(1-\varepsilon_0)^{-\frac1p}$ and obtain
\begin{equation*}
R_2^{\eta+1}\dfrac{\exp((1-\varepsilon_0)^{\frac{1}{1-p}}\beta_{0,2}|u^*_\eta(R_2)|^{\frac{p}{p-1}})}{|u^*_\eta(R_2)|^{\frac{p}{p-1}}}\leq C(1-\varepsilon_0)^{\frac{2-p}{p-1}}\int_{R_2}^\infty |u^*_\eta(r)|^pr^\eta\mathrm dr.
\end{equation*}
Therefore, we have
\begin{equation*}
\int_0^{R_2}\dfrac{\exp_p(\beta_{0,2}|u^*_\eta|^{\frac{p}{p-1}})}{(1+|u^*_\eta|)^{\frac{p}{p-1}}}r^\eta\mathrm dr\leq C\|u^*_\eta\|_{L^p_\eta}^p.
\end{equation*}
 \end{proof}

The proof of Theorem \ref{theok1} follows a similar approach to the proof of Theorem \ref{theo1}, but with a slight modification. Instead of relying on Proposition \ref{prop22}, we derive a similar result as \eqref{eqpt2} through the following expression:
\begin{equation}\label{eqk1}
u^*_\eta(r_1)-u^*_\eta(r_2)=-\int_{r_1}^{r_2}(u^*_\eta)'(r)\mathrm dr\leq\left(\int_{r_1}^{r_2}\left|(u^*_\eta)'\right|^pr^{p-1}\mathrm dr\right)^{\frac{1}{p}}\left(\log\dfrac{r_2}{r_1}\right)^{\frac{p-1}{p}},
\end{equation}
for all $0<r_1\leq r_2$. Applying the P\'olya-Szeg\"o inequality for weighted Sobolev spaces, as given by \cite[Theorem 3.2]{MR4097244}, to $u\in X^{1,p}_\infty(\eta,\alpha_1)$, we obtain
\begin{equation}
\int_0^ \infty|(u_\eta^*)'|^pr^{p-1}\mathrm dr\leq\int_0^\infty|u'|^pr^{p-1}\mathrm dr\leq1.
\end{equation}
Furthermore, by employing the same argument as in the proof of Lemma \ref{lemmakey}, but using \eqref{eqk1} instead of \eqref{eqpt2}, we can establish the following lemma:

 \begin{lemma}\label{lemmakey1}
Assume $\alpha_1=p-1>1$, $\alpha_0\geq-1$, and $\eta>-1$. Let $u\in X^{1,p}_\infty(\alpha_0,\alpha_1)$ with compact support in $[0,\infty)$. If $u^*_\eta(R)\geq1$ and
\begin{equation*}
\int_{ R}^\infty|u^{*}_\eta|^pr^{p-1}\mathrm dr\leq 1,
\end{equation*}
then
\begin{equation*}
\dfrac{\exp(\beta_{0,1}|u^*_\eta(R)|^{\frac{p}{p-1}})}{|u^*_\eta(R)|^{\frac{p}{p-1}}}R^{\eta+1}\leq C\int_R^\infty |u^*_\eta(r)|^pr^\eta\mathrm dr,
\end{equation*}
where $C=C(p,\eta)>0$ and $u^*_\eta$ is the half $\mu_{\eta,\eta}$-symmetrization.
 \end{lemma}

\section{Higher Order Derivative Case}\label{sectionk}

In this section, we demonstrate Adams' inequality with exact growth for higher order derivatives (Theorem \ref{theoexactk}) by utilizing the second order derivative case (Theorem \ref{theo1}). Prior to proving the main theorem, we first establish three lemmas. The first lemma (Lemma \ref{lemmahardycons}) provides a bound on the norm $\|u\|_{L^p_{\alpha-p}}$ in terms of the norm $\|u'\|_{L^p_\alpha}$, while the second lemma (Lemma \ref{lemmahardyL}) estimates in terms of $\|L_{\theta,\gamma}u\|_{L^p_{\nu}}$. Finally, the third lemma (Lemma \ref{lemmaLjnorm}) employs the estimate from the second lemma to establish a bound on $\|L_{\theta,\gamma}u\|_{L^p_\alpha}$ based on $\|L^j_{\theta,\gamma}u\|_{L^p_\nu}$. Towards the end of the section, we prove the Corollaries \ref{cor1} and \ref{cor2}.

By employing these lemmas, we prove the main theorem, which establishes Adams' inequality with exact growth for higher order derivatives. We now proceed to present the three lemmas, followed by the proof of the main theorem.

\begin{lemma}\label{lemmahardycons}
Let $u\in AC_{\mathrm{loc}}(0,R)$ with $\lim_{r\to R}u(r)=0$. If $\alpha-p+1>0$, then
\begin{equation*}
\left(\int_0^R|u|^pr^{\alpha-p}\mathrm dr\right)^{\frac1p}\leq\dfrac{p}{\alpha-p+1}\left(\int_0^R|u'|^pr^\alpha\mathrm dr\right)^{\frac1p}.
\end{equation*}
\end{lemma}
\begin{proof}
According to \cite[Theorem 6.2]{MR1069756}, it suffices to verify that
\begin{equation*}
\dfrac{p}{\alpha-p+1}=\dfrac{p}{(p-1)^{\frac{p-1}{p}}}\sup_{0<r<R}\|r^{\frac{\alpha-p}p}\|_{L^p(0,r)}\|r^{-\frac{\alpha}p}\|_{L^{\frac{p}{p-1}}(r,R)}.
\end{equation*}
The proof follows because
\begin{equation*}
\|r^{\frac{\alpha-p}p}\|_{L^p(0,r)}\|r^{-\frac{\alpha}p}\|_{L^{\frac{p}{p-1}}(r,R)}=\dfrac{(p-1)^{\frac{p-1}p}}{\alpha-p+1}\left(1-R^{-\frac{\alpha-p+1}{p-1}}r^{(\alpha-p+1)\frac{p}{p-1}}\right)^{\frac{p}{p-1}}.
\end{equation*}
\end{proof}

\begin{lemma}\label{lemmahardyL}
Let $p>1$, $R\in(0,\infty)$, $\alpha>-1$, and the elliptic operator $L_{\theta,\gamma}u=-r^{-\theta}(r^\gamma u')'$ with $\gamma>1$ and $\alpha+1<p(\gamma-1)$. For any $u\in AC^1_{\mathrm{loc}}(0,R)$ such that $\lim_{r\to R}u(r)=\lim_{r\to0}r^\gamma u'(r)=0$, we have
\begin{equation*}
\left(\int_0^R|u|^pr^\alpha\mathrm dr\right)^{\frac1p}\leq C_{\gamma,\alpha,p}\left(\int_0^R|L_{\theta,\gamma}u|^pr^{p(\theta+2-\gamma)+\alpha}\mathrm dr\right)^{\frac1p},
\end{equation*}
where
\begin{equation*}
C_{\gamma,\alpha,p}=\dfrac{p^2}{(\alpha+1)[p(\gamma-1)-\alpha-1]}.
\end{equation*}
\end{lemma}
\begin{proof}
Let $w(t)=u(Rt^{-\frac{1}{\gamma-1}})$. Then
\begin{equation}\label{eqita3}
\int_0^R|u|^pr^\alpha\mathrm dr=\dfrac{R^{\alpha+1}}{\gamma-1}\int_1^\infty|w(t)|^pt^{-\frac{\alpha+\gamma}{\gamma-1}}\mathrm dt
\end{equation}
and
\begin{equation}\label{eqita4}
\int_0^R|L_{\theta,\gamma}u|^pr^{p(\theta+2-\gamma)+\alpha}\mathrm dr=(\gamma-1)^{2p-1}R^{\alpha+1}\int_1^\infty|w''(t)|^pt^{\frac{2p\gamma-\gamma-2p-\alpha}{\gamma-1}}\mathrm dt.
\end{equation}
Since $\lim_{r\to R}u(r)=\lim_{r\to0}r^\gamma u'(r)=0$, we have
\begin{equation*}
w(t)=\int_1^t\int_z^\infty-w''(s)\mathrm ds\mathrm dz.
\end{equation*}
Set
\begin{equation*}
a=\frac{p-1}{p^2}\left(2p-\frac{\alpha+1}{\gamma-1}\right).
\end{equation*}
From the conditions $\alpha>-1$, $\gamma>1$, $p>1$, and $\alpha+1<p(\gamma-1)$, we obtain that
\begin{equation}\label{eqita}
1<\dfrac{ap}{p-1}<2
\end{equation}
and
\begin{equation}\label{eqita2}
2p-1-\dfrac{\alpha+\gamma}{\gamma-1}<ap<2p-\dfrac{\alpha+\gamma}{\gamma-1}.
\end{equation}
Thus, by \eqref{eqita},
\begin{align*}
|w(t)|^p&=\left(\int_1^t\int_z^\infty-w''(s)\frac{s^a}{s^a}\mathrm ds\mathrm dz\right)^p\\
&\leq\left[\left(\int_1^t\int_z^\infty|w''(s)|^ps^{ap}\mathrm ds\mathrm dz\right)^{\frac1p}\left(\int_1^t\int_z^\infty s^{-\frac{ap}{p-1}}\mathrm ds\mathrm dz\right)^{\frac{p-1}p}\right]^p\\
&=\int_1^t\int_z^\infty|w''(s)|^ps^{ap}\mathrm ds\mathrm dz\left(\frac{p(\gamma-1)}{p(\gamma-1)-\alpha-1}\int_1^tz^{1-\frac{ap}{p-1}}\mathrm dz\right)^{p-1}\\
&\leq\left(\dfrac{p^2(\gamma-1)^2}{(\alpha+1)[p(\gamma-1)-\alpha-1]}\right)^{p-1}\int_1^t\int_z^\infty|w''(s)|^ps^{ap}\mathrm ds\mathrm dzt^{2p-2-ap},
\end{align*}
for all $t>0$. Using \eqref{eqita2}, we have
\begin{align*}
\int_1^\infty&|w(t)|^pt^{-\frac{\alpha+\gamma}{\gamma-1}}\mathrm dt\\
&\leq\left(\dfrac{p^2(\gamma-1)^2}{(\alpha+1)[p(\gamma-1)-\alpha-1]}\right)^{p-1}\int_1^\infty\int_1^t\int_z^\infty|w''(s)|^ps^{ap}t^{2p-2-\frac{\alpha+\gamma}{\gamma-1}-ap}\mathrm ds\mathrm dz\mathrm dt\\
&=\left(\dfrac{p^2(\gamma-1)^2}{(\alpha+1)[p(\gamma-1)-\alpha-1]}\right)^{p-1}\int_1^\infty|w''(s)|^ps^{ap}\int_1^s\int_z^\infty t^{2p-2-\frac{\alpha+\gamma}{\gamma-1}-ap}\mathrm dt\mathrm dz\mathrm ds\\
&\leq\left(\dfrac{p^2(\gamma-1)^2}{(\alpha+1)[p(\gamma-1)-\alpha-1]}\right)^{p}\int_1^\infty|w''(s)|^ps^{2p-\frac{\alpha+\gamma}{\gamma-1}}\mathrm ds.
\end{align*}
Therefore, using \eqref{eqita3} and \eqref{eqita4}, we obtain
\begin{equation*}
\int_0^R|u|^pr^\alpha\mathrm dr\leq\left(\dfrac{p^2}{(\alpha+1)[p(\gamma-1)-\alpha-1]}\right)^{p} \int_0^R|L_{\theta,\gamma}u|^pr^{p(\theta+2-\gamma)+\alpha}\mathrm dr.
\end{equation*}
\end{proof}

The following lemma was also obtained in \cite[Lemma 4.8]{arXiv:2302.02262} for the particular case $\gamma=\theta$.

\begin{lemma}\label{lemmaLjnorm}
    Let $p,\gamma>1,\alpha,\theta\in\mathbb R$, and $j\geq2$ be an integer such that
    \begin{equation}\label{eqLjnorm}
    \left\{\begin{array}{ll}
         -1<\alpha<p(\gamma-1)+(2-j)(\theta+2-\gamma)p-1,&\mbox{if }\theta+2\geq\gamma,  \\
         (2-j)(\theta+2-\gamma)p-1<\alpha<p(\gamma-1)-1,&\mbox{if }\theta+2<\gamma.
    \end{array}\right.
    \end{equation}
    Suppose $u\in AC_{\mathrm{loc}}^{2j-1}(0,R)$ with $\lim_{r\to R}L_{\theta,\gamma}^iu(r)=\lim_{r\to0}r^\gamma(L_{\theta,\gamma}^iu)'(r)=0$ for all $i=1,\ldots,j-1$. Then
    \begin{equation*}
    \|L_{\theta,\gamma}u\|_{L^p_{\alpha}}\leq\left(\prod_{i=1}^{j-1}C_i\right)\|L^j_{\theta,\gamma}u\|_{L^p_{\alpha+(j-1)(\theta+2-\gamma)p}},
    \end{equation*}
    where, for each $i=1,\ldots,j-1$,
    \begin{equation*}
    C_i=\dfrac{p^2}{\left[\alpha+(i-1)(\theta+2-\gamma)p+1\right]\left[p(\gamma-1)-\alpha-1-(i-1)(\theta+2-\gamma)p\right]}.
    \end{equation*}
\end{lemma}
\begin{proof}
Note that \eqref{eqLjnorm} is equivalent to
\begin{equation*}
0<\alpha+(i-1)(\theta+2-\gamma)p+1<p(\gamma-1),\quad\forall i=1,\ldots,j-1.
\end{equation*}
Therefore, we can apply Lemma \ref{lemmahardyL} $j-1$ times to conclude the result.
\end{proof}

\begin{proof}[Proof of Theorem \ref{theoexactk}]
From Lemma \ref{lemmacompsuppdense} and Lemma \ref{lemmajaosn} (using $\alpha_i\geq\alpha_k-(k-i)p$ for all $i=1,\ldots,k$), we can assume $\mathrm{supp}(u)\subset[0,\infty)$ is compact and $u\in C^\infty(0,\infty)$. Let $j$ be such that $k=2j$ or $k=2j+1$ ($j=\lfloor\frac{k}2\rfloor$). Since $\theta+2>\gamma$ and $\theta>j(\theta+2-\gamma)-1$, we can apply Lemma \ref{lemmaLjnorm} to obtain
\begin{equation}\label{eqbessa}
\|L_{\theta,\gamma} u\|_{L^p_{2p-1+(\theta-\gamma)p}}\leq\left(\prod_{i=1}^{j-1}C_i\right)\|L_{\theta,\gamma}^{j}u\|_{L^p_{2jp-1+j(\theta-\gamma)p}},
\end{equation}
where
\begin{equation*}
C_i=\dfrac{1}{i(\theta+2-\gamma)[\gamma-1-i(\theta+2-\gamma)]},\quad\forall i=1,\ldots,j-1.
\end{equation*}
Now we slit the proof into two cases: $k=2j$ and $k=2j+1$.

\vspace{0.3cm}

\noindent \underline{Case $k=2j$:} We define $v:=\prod_{i=1}^{j-1}C_i^{-1}u$. From \eqref{eqbessa} and $\|\nabla^k_L u\|_{L^p_\nu}\leq1$, we have $\|L_{\theta,\gamma}v\|_{L^p_{2p-1}}\leq1$. Applying Theorem \ref{theo1} to $v$, we obtain
\begin{equation*}
\int_0^\infty\dfrac{\exp_p(\beta_{0,2}|v|^{\frac{p}{p-1}})}{(1+|v|)^{\frac{p}{p-1}}}r^\eta\mathrm dr\leq C\|v\|^p_{L^p_{\eta}}.
\end{equation*}
Thus,
\begin{equation*}
\int_0^\infty\dfrac{\exp_p(\beta_{0,2}\prod_{i=1}^{j-1}C_i^{-\frac{p}{p-1}}|u|^{\frac{p}{p-1}})}{(1+|u|)^{\frac{p}{p-1}}}r^\eta\mathrm dr\leq C\|u\|^p_{L^p_{\eta}}.
\end{equation*}
We are left with the task of proving
\begin{equation}\label{eqkg1}
(\eta+1)\left[(\gamma-1)(\theta+2-\gamma)^{k-2}\dfrac{\Gamma\left(\frac{k}2\right)\Gamma\left(\frac{\gamma-1}{\theta+2-\gamma}\right)}{\Gamma\left(\frac{\gamma-1}{\theta+2-\gamma}-\frac{k-2}{2}\right)}\right]^{\frac{p}{p-1}}=\beta_{0,2}\prod_{i=1}^{j-1}C_i^{-\frac{p}{p-1}}.
\end{equation}
Note that
\begin{align}
\prod_{i=1}^{j-1}C_i^{-1}&=(j-1)!(\theta+2-\gamma)^{j-1}\prod_{i=1}^{j-1}[\gamma-1-i(\theta+2-\gamma)]\nonumber\\
&=\Gamma(j)(\theta+2-\gamma)^{2j-2}\prod_{i=1}^{j-1}\left(\frac{\gamma-1}{\theta+2-\gamma}-i\right)\nonumber\\
&=(\theta+2-\gamma)^{2j-2}\dfrac{\Gamma(j)\Gamma\left(\frac{\gamma-1}{\theta+2-\gamma}\right)}{\Gamma\left(\frac{\gamma-1}{\theta+2-\gamma}-j+1\right)}.\label{eqkg2}
\end{align}
Using the expression of $\beta_{0,2}$ together with $j=\frac{k}{2}$ and \eqref{eqkg2}, we conclude \eqref{eqkg1} as desired.

\vspace{0.3cm}

\noindent \underline{Case $k=2j+1$:} We define $v:=j(\theta+2-\gamma)\prod_{i=1}^{j-1}C_i^{-1}u$. From \eqref{eqbessa}, Lemma \ref{lemmahardycons}, and $\|\nabla^k_L u\|_{L^p_\nu}\leq1$, we have
\begin{equation*}
\|L_{\theta,\gamma}v\|_{L^p_{2p-1}}\leq j(\theta+2-\gamma)\|L^j_{\theta,\gamma}u\|_{L^p_{2jp-1+j(\theta-\gamma)p}}\leq \|\left(L^j_{\theta,\gamma}u\right)'\|_{L^p_{\alpha_k+\lfloor\frac{k}{2}\rfloor(\theta-\gamma)p}}\leq1.
\end{equation*}
Applying Theorem \ref{theo1} to $v$, we obtain
\begin{equation*}
\int_0^\infty\dfrac{\exp_p(\beta_{0,2}|v|^{\frac{p}{p-1}})}{(1+|v|)^{\frac{p}{p-1}}}r^\eta\mathrm dr\leq C\|v\|^p_{L^p_{\eta}}.
\end{equation*}
Thus,
\begin{equation*}
\int_0^\infty\dfrac{\exp_p(\beta_{0,2}j^{\frac{p}{p-1}}(\theta+2-\gamma)^{\frac{p}{p-1}}\prod_{i=1}^{j-1}C_i^{-\frac{p}{p-1}}|u|^{\frac{p}{p-1}})}{(1+|u|)^{\frac{p}{p-1}}}r^\eta\mathrm dr\leq C\|u\|^p_{L^p_{\eta}}.
\end{equation*}
Using the expression of $\beta_{0,2}$, we are left with the task of proving
\begin{equation}\label{eqkg3}
(\theta+2-\gamma)^{k-2}\dfrac{\Gamma\left(\frac{k+1}2\right)\Gamma\left(\frac{\gamma-1}{\theta+2-\gamma}\right)}{\Gamma\left(\frac{\gamma-1}{\theta+2-\gamma}-\frac{k-3}{2}\right)}=j(\theta+2-\gamma)\prod_{i=1}^{j-1}C_i^{-1}.
\end{equation}
By \eqref{eqkg2} and $j=\frac{k-1}2$, we obtain
\begin{align*}
j(\theta+2-\gamma)\prod_{i=1}^{j-1}C_i^{-1}&=(\theta+2-\gamma)^{2j-1}\dfrac{\Gamma(j+1)\Gamma\left(\frac{\gamma-1}{\theta+2-\gamma}\right)}{\Gamma\left(\frac{\gamma-1}{\theta+2-\gamma}-j+1\right)}\\
&=(\theta+2-\gamma)^{k-2}\dfrac{\Gamma\left(\frac{k+1}2\right)\Gamma\left(\frac{\gamma-1}{\theta+2-\gamma}\right)}{\Gamma\left(\frac{\gamma-1}{\theta+2-\gamma}-\frac{k-3}{2}\right)},
\end{align*}
which concludes \eqref{eqkg3} and the proof of Theorem \ref{theoexactk}.
\end{proof}

\begin{proof}[Proof of Corollary \ref{cor1}]
The proof directly follows from the inequality:
\begin{equation}\label{erqgnal}
\exp(\beta t^{\frac{p}{p-1}})\leq C_\beta\dfrac{\exp_p(\beta_{0,k}t^{\frac{p}{p-1}})}{(1+t)^{\frac{p}{p-1}}}\quad\forall \beta\in[0,\beta_{0,k}), t\geq0.
\end{equation}
To prove the inequality \eqref{erqgnal}, we define $\phi(t)=(1+t)^{\frac{p}{p-1}}\exp(\beta t^{\frac{p}{p-1}})/\exp_p(\beta_{0,k}t^{\frac{p}{p-1}})$ for $t>0$ and use the fact that $\phi(t)\overset{t\to0}\longrightarrow\beta/\beta_{0,k}$ and $\phi(t)\overset{t\to\infty}{\longrightarrow}0$ (see \cite[Lemma 3.1]{arXiv:2306.00194}).
\end{proof}

\begin{proof}[Proof of Corollary \ref{cor2}]
Let $u\in X^{k,p}_\infty\backslash\{0\}$ be such that $\|\nabla^k_Lu\|^p_{L^p_\nu}+\tau\|u\|^p_{L^p_\eta}\leq1$. Choose $\theta\in(0,1)$ with $\|u\|^p_{L^p_\eta}=\theta$ and $\|\nabla^k_Lu\|^p_{L^p_\nu}\leq1-\tau\theta$. In particular, $\theta\leq \frac{1}{\tau}$.

\vspace{0.3cm}

\noindent \underline{Case $\theta\geq\frac{p-1}{\tau p}$:} Define $v=p^{\frac1p}u$. Then $\|\nabla^k_Lv\|^p_{L^p_\nu}=p\|\nabla^k_Lu\|^p_{L^p_\nu}\leq p(1-\tau \theta)\leq1$. By Corollary \ref{cor1}, we have
\begin{equation*}
\int_0^\infty\exp_p\left(p^{\frac1{p-1}}\beta|u|^{\frac{p}{p-1}}\right)r^\eta\mathrm dr\leq C_\beta,\quad\forall \beta\in(0,\beta_{0,k}).
\end{equation*}
Since $0<\beta_{0,k}p^{-\frac{1}{p-1}}<\beta_{0,k}$, we can choose $\beta=\beta_{0,k}p^{-\frac{1}{p-1}}$ in the above inequality to conclude this case.

\vspace{0.3cm}

\noindent \underline{Case $\theta<\frac{p-1}{\tau p}$:} Let $A:=\{r>0\colon|u(r)|\geq1\}$. Since $|u(r)|<1$ for $r\in(0,\infty)\backslash A$ and $\exp_p(t)\leq Ct^{p-1}$ for all $t\in[0,\beta_{0,k}]$, we have
\begin{equation*}
\int_{(0,\infty)\backslash A}\exp_p\left(\beta_{0,k}|u|^{\frac{p}{p-1}}\right)r^\eta\mathrm dr\leq C\beta_{0,k}^{p-1}\int_{(0,\infty)\backslash A}|u|^pr^\eta\mathrm dr\leq C\beta_{0,k}^{p-1}.
\end{equation*}
We now are left with the task to prove that
\begin{equation}\label{eqc2}
\int_A\exp_p\left(\beta_{0,k}|u|^{\frac{p}{p-1}}\right)r^\eta\mathrm dr\leq C.
\end{equation}
Note that $\left[\exp_p(t)\right]^q\leq\exp_p(qt)$ for all $t\geq0$ and $q\geq1$. This follows from the fact that $t\mapsto\left[\exp_p(t)\right]^q/\exp_p(qt)$ is a nonincreasing function that converges to 1 as $t$ goes to infinity (see \cite[Lemma 3.1]{arXiv:2306.00194}). Using H\"older's inequality with $q=\frac{p-1}{p-1-\tau\theta}>1$, we obtain
\begin{align}
\int_A\exp_p\left(\beta_{0,k}|u|^{\frac{p}{p-1}}\right)r^\eta\mathrm dr&\leq\!\left(\!\int_A\dfrac{\exp_p\left(q\beta_{0,k}|u|^{\frac{p}{p-1}}\right)}{(1+|u|)^{\frac{p}{p-1}}}r^\eta\mathrm dr\!\right)^{\frac1q}\!\left(\!\int_A\!(1+|u|)^{\frac{p}{(p-1)(q-1)}}r^\eta\mathrm dr\!\right)^{\frac{q-1}q}\nonumber\\
&\leq2^{\frac{p}{(p-1)q}}\left(q\int_A\dfrac{\exp_p\left(\beta_{0,k}|v|^{\frac{p}{p-1}}\right)}{(1+|v|)^{\frac{p}{p-1}}}r^\eta\mathrm dr\right)^{\frac1q}\left\||u|^{\frac{p}{p-1}}\right\|^{\frac{1}q}_{L^{\frac{1}{q-1}}_\eta}\label{eqc4},
\end{align}
where $v=q^{\frac{p-1}p}u$. From Corollary \ref{cor1}, we have
\begin{equation}\label{eqc3}
\left\||u|^{\frac{p}{p-1}}\right\|_{L^s_\eta}\leq Cs\|u\|^{\frac{p}s}_{L^p_\eta},\quad\forall u\in X^{k,p}_\infty\mbox{ with }\|\nabla^k_Lu\|_{L^p_\nu}\leq1,
\end{equation}
for all $s$ in the form of $p-1+n$ with $n\in\mathbb N\cup\{0\}$. Using interpolation of weighted Lebesgue spaces, we can extend this inequality for all $s\geq p-1$. Then, by \eqref{eqc3},
\begin{equation}\label{eqc5}
\left\||u|^{\frac{p}{p-1}}\right\|^{\frac1q}_{L^{\frac{1}{q-1}}_\eta}\leq C\left(\dfrac{1}{q-1}\right)^{\frac1q}\|u\|_{L^p_\eta}^{\frac{p(q-1)}q}.
\end{equation}
From the definition of $q$, $\theta\leq\frac{1}\tau$, $p\geq2$, and $(1-x)^s\leq1-sx$ for all $s,x\in[0,1]$, we have
\begin{equation*}
\|\nabla^k_Lv\|^p_{L^p_\eta}\leq q^{p-1}(1-\tau\theta)=\left[\dfrac{(1-\tau\theta)^{\frac{1}{p-1}}}{1-\frac{\tau\theta}{p-1}}\right]^{p-1}\leq1.
\end{equation*}
Thus we can apply Theorem \ref{theoexactk} to $v\in X^{k,p}_\infty$ to obtain
\begin{equation*}
\int_0^\infty\dfrac{\exp_p\left(q\beta_{0,k}|v|^{\frac{p}{p-1}}\right)}{(1+|v|)^{\frac{p}{p-1}}}r^\eta\mathrm dr\leq Cq^{p-1}\|u\|^p_{L^p_\eta}.
\end{equation*}
Using this along with \eqref{eqc4} and \eqref{eqc5}, we obtain
\begin{align*}
\int_A\exp_p\left(\beta_{0,k}|u|^{\frac{p}{p-1}}\right)r^\eta\mathrm dr&\leq Cq^{\frac{p}q}\|u\|^{\frac{p}{q}}_{L^p_\eta}\left(\dfrac{1}{q-1}\right)^{\frac1q}\|u\|_{L^p_\eta}^{\frac{p(q-1)}q}=Cq^{\frac{p-1}q}\left(\frac{q}{q-1}\right)^{\frac1q}\theta^p\\
&=C\tau^{\frac{\tau\theta}{p-1}-1}(p-1)^{\frac{p-1-\tau\theta}{p-1}}\dfrac{\theta^{p-1+\frac{\tau\theta}{p-1}}}{\left(1-\frac{\tau\theta}{p-1}\right)^{p-1-\tau\theta}}.
\end{align*}
This concludes \eqref{eqc2} and the proof because the right term is bounded on $0<\theta<\frac{p-1}{\tau p}$.
\end{proof}

\section{Proof of the Supercritical Case}\label{sectionsuperc}

In this final section, we focus on the behavior of the supremum in two scenarios: when $\beta>\beta_{0,k}$ and when $\beta=\beta_{0,k}$ while the power $q$ in the denominator satisfies $q<\frac{p}{p-1}$. More specifically, we provide the proof of Theorem \ref{theosuperc}, establishing the sharpness of the constant $\beta_{0,k}$ and the exponent $\frac{p}{p-1}$.

\begin{proof}[Proof of Theorem \ref{theosuperc}]
The construction of the sequence follows a similar argument to that used by D. R. Adams in \cite{MR0960950}. Let $\phi\in C^\infty[0,1]$ be such that
\begin{equation*}
\left\{\begin{array}{ll}
\phi(0)=\phi'(0)=\cdots=\phi^{(k+1)}(0)=0,\ \phi'\geq0,\\
\phi(1)=\phi'(1)=1,\ \phi''(1)=\cdots\phi^{(k-1)}(1)=0.
\end{array}\right.
\end{equation*}
Consider $0<\varepsilon<1/2$ and define
\begin{equation*}
H(t)=\left\{\begin{array}{llll}
     \varepsilon\phi\left(\dfrac{t}{\varepsilon}\right),&\mbox{if }0<t\leq\varepsilon  \\
     t,&\mbox{if }\varepsilon<t\leq1-\varepsilon,\\
     1-\varepsilon\phi\left(\dfrac{1-t}{\varepsilon}\right),&\mbox{if }1-\varepsilon<t\leq1,\\
     1,&\mbox{if }t>1.
\end{array}\right.
\end{equation*}
Let $n\in\mathbb N$, $R\in(0,\infty)$ fixed, and define the sequence $(\psi_{n,\varepsilon})_n$ as follows
\begin{equation*}
\psi_{n,\varepsilon}(r)=H\left(\dfrac{\log\frac{R}{r}}{\log n}\right),\quad 0<r<R.
\end{equation*}
It is not hard to see that $\psi_{n,\varepsilon}\in X^{k,p}_{0,R}\subset X^{k,p}_{\infty}$. By induction on $m\in\mathbb N$ we have
\begin{equation}\label{eqsc0}
L^m_{\theta,\gamma}\psi_{n,\varepsilon}(r)=r^{m(\gamma-2-\theta)}\sum_{i=1}^{2m}\dfrac{c_{im}}{(\log n)^i}H^{(i)}\left(\dfrac{\log\frac{R}{r}}{\log n}\right),
\end{equation}
where
\begin{equation*}
\left\{\begin{array}{l}
c_{11}=\gamma-1,\ c_{21}=-1;\\
c_{1m+1}=-m(\gamma-2-\theta)(\gamma-1+m(\gamma-2-\theta))c_{1m};  \\
     c_{2m+1}=-m(\gamma-2-\theta)(\gamma-1+m(\gamma-2-\theta))c_{2m}+(\gamma-1+2m(\gamma-2-\theta))c_{1m};\\
     c_{2m+1m+1}=(\gamma-1+2m(\gamma-2-\theta))c_{2mm}-c_{2m-1m}\\
     c_{2m+2m+1}=-c_{2mm}.
\end{array}\right.
\end{equation*}
and, for each $i=3,\ldots,2m$,
\begin{equation*}
c_{im+1}=-m(\gamma-2-\theta)(\gamma-1+m(\gamma-2-\theta))c_{im}+(\gamma-1+2m(\gamma-2-\theta))c_{i-1m}-c_{i-2m}.
\end{equation*}
From $c_{11}=\gamma-1$ and $c_{1m+1}=-m(\gamma-2-\theta)(\gamma-1+m(\gamma-2-\theta))c_{1m}$ we obtain
\begin{align}
c_{1m}&=(\gamma-1)\prod_{i=1}^{m-1}\left[-i(\gamma-2-\theta)(\gamma-1+i(\gamma-2-\theta))\right]\nonumber\\
&=(\gamma-1)(\theta+2-\gamma)^{m-1}(m-1)!\prod_{i=1}^{m-1}(\theta+2-\gamma)\left(\dfrac{\gamma-1}{\theta+2-\gamma}-i\right)\nonumber\\
&=(\gamma-1)(\theta+2-\gamma)^{2m-2}\Gamma(m)\dfrac{\Gamma\left(\frac{\gamma-1}{\theta+2-\gamma}\right)}{\Gamma\left(\frac{\gamma-1}{\theta+2-\gamma}-m+1\right)},\quad\forall m\in\mathbb N.\label{eqsc1}
\end{align}

We make the following claim:
\begin{equation}\label{eq44}
\|\psi_{n,\varepsilon}\|_{L^p_{\eta}}^p=O\left((\log n)^{-p}\right).
\end{equation}
To prove this, we begin by introducing a change of variables $t=\log\frac{R}{r}/\log n$, which gives us:
\begin{align}
\|\psi_{n,\varepsilon}&\|_{L^p_{\eta}}^p=\int_0^R\left|H\left(\dfrac{\log\frac{R}{r}}{\log n}\right)\right|^pr^{\eta}\mathrm dr\nonumber\\
&=R^{\eta+1}\log n\Bigg[\int_0^\varepsilon\left|\varepsilon\phi\left(\dfrac{t}\varepsilon\right)\right|^pn^{-(\eta+1)t}\mathrm dt+\int_\varepsilon^{1-\varepsilon}t^pn^{-(\eta+1)t}\mathrm dt\nonumber\\
&\quad+\int_{1-\varepsilon}^1\left|1-\varepsilon\phi\left(\dfrac{1-t}{\varepsilon}\right)\right|^pn^{-(\eta+1)t}\mathrm dt+\int_1^\infty n^{-(\eta+1)t}\mathrm dt\Bigg]\nonumber\\
&\leq R^{\eta+1}\log n\left[\int_0^\varepsilon\left|\varepsilon\phi\left(\dfrac{t}\varepsilon\right)\right|^pn^{-(\eta+1)t}\mathrm dt+\int_\varepsilon^\infty n^{-(\eta+1)t}\mathrm dt\right]\nonumber\\
&=R^{\eta+1}\log n\int_0^\varepsilon\left|\varepsilon\phi\left(\dfrac{t}\varepsilon\right)\right|^pn^{-(\eta+1)t}\mathrm dt+O\left((\log n)^{-p}\right).\label{eq451}
\end{align}
Let $\nu\in (\frac{1}{k+1},1)$ fixed. Since $\phi\left((\log n)^{-\nu}\right)\leq (\log n)^{-\nu (k+1)}$ for large $n$, we can estimate the integral (for large $n$) as follows
\begin{align}
\int_0^\varepsilon\left|\varepsilon\phi\left(\dfrac{t}\varepsilon\right)\right|^pn^{-(\eta+1)t}\mathrm dt&\leq\varepsilon^p\int_0^{\varepsilon(\log n)^{-\nu}}\left|\phi\left((\log n)^{-\nu}\right)\right|^pn^{-(\eta+1)t}\mathrm dt\nonumber\\
&\qquad+\varepsilon^p\int_{\varepsilon(\log n)^{-\nu}}^\varepsilon n^{-(\eta+1)t}\mathrm dt\nonumber\\
&\leq \varepsilon^p(\log n)^{-\nu p(k+1)}\dfrac1{(\eta+1)\log n}\nonumber\\
&\qquad+\varepsilon^p\dfrac{1}{(\eta+1)e^{\varepsilon(\eta+1)(\log n)^{1-\nu}}\log n}\nonumber\\
&=O\left((\log n)^{-p-1}\right).\label{eq46}
\end{align}
Using \eqref{eq451} along with \eqref{eq46}, we conclude \eqref{eq44}.

\vspace{0.3cm}

\noindent \underline{Case $k=2j$:} By equation \eqref{eqsc0}, we have
\begin{equation*}
    L^j_{\theta,\gamma}\psi_{n,\varepsilon}(r)=r^{j(\gamma-2-\theta)}\frac{c_{1j}}{\log n}H'\left(\dfrac{\log\frac{R}{r}}{\log n}\right)+r^{j(\gamma-2-\theta)}O((\log n)^{-2}).
\end{equation*}
Then,
\begin{align}
\|\nabla_L^k\psi_{n,\varepsilon}\|^p_{L^p_{\alpha_k+j(\theta-\gamma)p}}&=\left|\dfrac{c_{1j}}{\log n}\right|^p\int_{\frac{R}n}^R\left|H'\left(\dfrac{\log\frac{R}{r}}{\log n}\right)+O\left((\log n)^{-1}\right)\right|^pr^{\alpha_k+j(\theta-\gamma)p+j(\gamma-2-\theta)p}\mathrm dr\nonumber\\
&\leq\left|\dfrac{c_{1j}}{\log n}\right|^p\Bigg[\int_{\frac{R}{n}}^{Rn^{\varepsilon-1}}\left|\|\phi'\|_\infty+O\left((\log n)^{-1}\right)\right|^pr^{-1}\mathrm dr\nonumber\\
&\quad+\int_{Rn^{\varepsilon-1}}^{Rn^{-\varepsilon}}r^{-1}\mathrm dr+\int_{Rn^{-\varepsilon}}^R\left|\|\phi'\|_{\infty}+O\left((\log n)^{-1}\right)\right|^pr^{-1}\Bigg]\nonumber\\
&\leq|c_{1j}|^p(\log n)^{1-p}\left[1+2\varepsilon\left|\|\phi'\|_\infty+O\left((\log n)^{-1}\right)\right|^p\right]\nonumber\\
&=|c_{1j}|^p(\log n)^{1-p}A_{\varepsilon,n},\label{eqsc2}
\end{align}
where $A_{\varepsilon,n}:=1+2\varepsilon\left|\|\phi'\|_\infty+O\left((\log n)^{-1}\right)\right|^p$. Set
\begin{equation*}
u_{n,\varepsilon}(r)=\dfrac{\psi_{n,\varepsilon}(r)}{|c_{1j}|(\log n)^{\frac{1-p}{p}}A_{\varepsilon,n}^{\frac{1}{p}}}.
\end{equation*}
Using \eqref{eqsc2} and \eqref{eq44}, we have $\|\nabla^k_Lu_{n,\varepsilon}\|_{L^p_{\alpha_k+j(\theta-\gamma)p}}\leq1$ and
\begin{equation}\label{gja}
\|u_{n,\varepsilon}\|^p_{L^p_\eta}=O((\log n)^{-1}).
\end{equation}
Given $q\geq0$ and $\beta\geq\beta_{0,k}$, equation \eqref{eqsc2} guarantees
\begin{align}
\int_0^\infty&\dfrac{\exp_p(\beta|u_{n,\varepsilon}|^{\frac{p}{p-1}})}{(1+|u_{n,\varepsilon}|)^{q}}r^\eta\mathrm dr\geq\int_0^{\frac{R}{n}}\dfrac{\exp_p(\beta|u_{n,\varepsilon}|^{\frac{p}{p-1}})}{(1+|u_{n,\varepsilon}|)^{q}}r^\eta\mathrm dr\nonumber\\
&\geq C\exp\left(\beta|c_{1j}|^{-\frac{p}{p-1}}A_{\varepsilon,n}^{\frac{1}{1-p}}\log n\right)\left(\log n\right)^{-q\frac{p-1}p}\frac{R^{\eta+1}}{n^{\eta+1}}\nonumber\\
&=C\exp\left[(\eta+1)\log n\left(\frac{\beta|c_{1j}|^{-\frac{p}{p-1}}A_{\varepsilon,n}^{\frac{1}{1-p}}}{\eta+1}-1\right)\right]\left(\log n\right)^{-q\frac{p-1}p}.\label{eq444}
\end{align}
From \eqref{eq444}, we can conclude \eqref{eqsuc2} because if assuming $\beta>(\eta+1)|c_{1j}|^{\frac{p}{p-1}}=\beta_{0,k}$, we can choose a small $\varepsilon>0$ such that $\beta|c_{1j}|^{-\frac{p}{p-1}}A_{\varepsilon,n}^{\frac{1}{1-p}}>\eta+1$. Now, let's assume $\beta=\beta_{0,k}$ and $q<p/(p-1)$ to prove \eqref{eqsuc1}. By using \eqref{gja} and \eqref{eq444}, we have
\begin{align*}
\dfrac{1}{\|u_{n,\varepsilon}\|_{L^p_{\eta}}^p}\int_0^\infty\dfrac{\exp_p(\beta_{0,k}|u_{n,\varepsilon}|^{\frac{p}{p-1}})}{(1+|u_{n,\varepsilon}|)^{q}}r^\eta\mathrm dr&\geq C(\log n)^{1-q\frac{p-1}p}\overset{n\to\infty}\longrightarrow\infty.
\end{align*}

\vspace{0.3cm}

\noindent \underline{Case $k=2j+1$:} This case is analogous to the previous case; however, by referring to equation \eqref{eqsc0}, we obtain
\begin{equation*}
    (L^j_{\theta,\gamma}\psi_{n,\varepsilon})'(r)=r^{j(\gamma-2-\theta)-1}\frac{j(\gamma-2-\theta)c_{1j}}{\log n}H'\left(\dfrac{\log\frac{R}{r}}{\log n}\right)+r^{j(\gamma-2-\theta)-1}O((\log n)^{-2}),
\end{equation*}
where we are denoting $j(\gamma-2-\theta)c_{1j}=1$ when $j=0$ to include the first derivative case in this theorem. Then,
\begin{align}
\|\nabla_L^k\psi_{n,\varepsilon}\|^p_{L^p_{\alpha_k+j(\theta-\gamma)p}}&=\left|\dfrac{j(\gamma-2-\theta)c_{1j}}{\log n}\right|^p\int_{\frac{R}n}^R\left|H'\left(\dfrac{\log\frac{R}{r}}{\log n}\right)+O\left((\log n)^{-1}\right)\right|^pr^{\alpha_k-2jp-p}\mathrm dr\nonumber\\
&\leq\left|\dfrac{j(\gamma-2-\theta)c_{1j}}{\log n}\right|^p\Bigg[\int_{\frac{R}{n}}^{Rn^{\varepsilon-1}}\left|\|\phi'\|_\infty+O\left((\log n)^{-1}\right)\right|^pr^{-1}\mathrm dr\nonumber\\
&\quad+\int_{Rn^{\varepsilon-1}}^{Rn^{-\varepsilon}}r^{-1}\mathrm dr+\int_{Rn^{-\varepsilon}}^R\left|\|\phi'\|_{\infty}+O\left((\log n)^{-1}\right)\right|^pr^{-1}\Bigg]\nonumber\\
&\leq|j(\gamma-2-\theta)c_{1j}|^p(\log n)^{1-p}\left[1+2\varepsilon\left|\|\phi'\|_\infty+O\left((\log n)^{-1}\right)\right|^p\right]\nonumber\\
&=|j(\gamma-2-\theta)c_{1j}|^p(\log n)^{1-p}A_{\varepsilon,n},\label{eqsc22}
\end{align}
where $A_{\varepsilon,n}:=1+2\varepsilon\left|\|\phi'\|_\infty+O\left((\log n)^{-1}\right)\right|^p$. Set
\begin{equation*}
u_{n,\varepsilon}(r)=\dfrac{\psi_{n,\varepsilon}(r)}{|j(\gamma-2-\theta)c_{1j}|(\log n)^{\frac{1-p}{p}}A_{\varepsilon,n}^{\frac{1}{p}}}.
\end{equation*}
Using \eqref{eqsc22} and \eqref{eq44}, we have $\|\nabla^k_Lu_{n,\varepsilon}\|_{L^p_{\alpha_k+j(\theta-\gamma)p}}\leq1$ and
\begin{equation}\label{gja2}
\|u_{n,\varepsilon}\|^p_{L^p_\eta}=O((\log n)^{-1}).
\end{equation}
Given $q\geq0$ and $\beta\geq\beta_{0,k}$, \eqref{eqsc22} guarantees
\begin{align}
\int_0^\infty&\dfrac{\exp_p(\beta|u_{n,\varepsilon}|^{\frac{p}{p-1}})}{(1+|u_{n,\varepsilon}|)^{q}}r^\eta\mathrm dr\geq\int_0^{\frac{R}{n}}\dfrac{\exp_p(\beta|u_{n,\varepsilon}|^{\frac{p}{p-1}})}{(1+|u_{n,\varepsilon}|)^{q}}r^\eta\mathrm dr\nonumber\\
&\geq C\exp\left(\beta|j(\gamma-2-\theta)c_{1j}|^{-\frac{p}{p-1}}A_{\varepsilon,n}^{\frac{1}{1-p}}\log n\right)\left(\log n\right)^{-q\frac{p-1}p}\frac{R^{\eta+1}}{n^{\eta+1}}\nonumber\\
&=C\exp\left[(\eta+1)\log n\left(\frac{\beta|j(\gamma-2-\theta)c_{1j}|^{-\frac{p}{p-1}}A_{\varepsilon,n}^{\frac{1}{1-p}}}{\eta+1}-1\right)\right]\left(\log n\right)^{-q\frac{p-1}p}.\label{eq4442}
\end{align}
From \eqref{eq4442}, we conclude \eqref{eqsuc2} because assuming $\beta>(\eta+1)|j(\gamma-2-\theta)c_{1j}|^{\frac{p}{p-1}}=\beta_{0,k}$, we can take $\varepsilon>0$ small with $\beta|j(\gamma-2-\theta)c_{1j}|^{-\frac{p}{p-1}}A_{\varepsilon,n}^{\frac{1}{1-p}}>\eta+1$. To establish \eqref{eqsuc1}, we will assume $\beta=\beta_{0,k}$ and $q<p/(p-1)$. Utilizing \eqref{gja2} and \eqref{eq4442}, we obtain
\begin{align*}
\dfrac{1}{\|u_{n,\varepsilon}\|_{L^p_{\eta}}^p}\int_0^\infty\dfrac{\exp_p(\beta_{0,k}|u_{n,\varepsilon}|^{\frac{p}{p-1}})}{(1+|u_{n,\varepsilon}|)^{q}}r^\eta\mathrm dr&\geq C(\log n)^{1-q\frac{p-1}p}\overset{n\to\infty}\longrightarrow\infty,
\end{align*}
which completes the proof of \eqref{eqsuc1} and therefore the proof of Theorem \ref{theosuperc}.
\end{proof}

\section{Application on an ODE}\label{secapp}

In this section, we work on the following problem
\begin{equation}\label{eqproblem}
    \left\{\begin{array}{l}
\Delta^2_{\theta}u=f(r,u)r^{\eta-\theta}\mbox{ in }(0,\infty),\\
u'(0)=(\Delta_\theta u)'(0)=0,
    \end{array}\right.
\end{equation}
where $\eta>-1$, $\theta=\gamma=(\eta+3)/2$, $\Delta_\theta:=L_{\theta,\theta}$, and $f\colon[0,\infty)\times\mathbb R\to\mathbb R$ continuous satisfying the growth condition
\begin{equation}\label{hipf}
|f(r,t)|\leq C\left(e^{\beta t^2}-1\right),\quad\mbox{for some }\beta>0.
\end{equation}
Specifically, we prove the existence of a weak solution for the problem \eqref{eqprobleml} and develop the regularity theory for the more general problem \eqref{eqproblem}.

Throughout this section, we consider the weighted Sobolev space $X^{2,2}_\infty(\alpha_0,\alpha_1,\alpha_2)$, where $\alpha_0\geq-1$, $\alpha_1\geq1$, and $\alpha_2=3$. Based on the following identity:
\begin{equation*}
\int_0^\infty \Delta_{\theta}uvr^\theta\mathrm dr=\int_0^\infty u'v'r^{\theta}\mathrm dr=\int_0^\infty u\Delta_{\theta}vr^\theta\mathrm dr,
\end{equation*}
for all $v\in X^{2,2}_\infty$ with compact support on $[0,\infty)$, we define $u\in X^{2,2}_\infty(\alpha_0,\alpha_1,\alpha_2)$ to be a weak solution of \eqref{eqproblem} if it satisfies
\begin{equation*}
\int_0^\infty \Delta_{\theta}u\Delta_{\theta}vr^{\theta}\mathrm dr=\int_0^\infty f(r,u)vr^\eta\mathrm dr,\quad\forall v\in X^{2,2}_\infty.
\end{equation*}
By $\theta=\gamma=(\eta+3)/2$ and Lemma \ref{lemmajaosn}, we can demonstrate that the left term is well-defined. On the other hand, the right term is well-defined by applying the inequalities $[\exp_2(t)]^p\leq\exp_2(pt)$, along with \eqref{hipf} and the following lemma:
\begin{lemma}\label{proptobeproved}
Consider $X^{k,p}_\infty(\alpha_0,\ldots,\alpha_k)$ with $p>1$ and $\alpha_i\geq\alpha_k-(k-i)p$ for all $i=0,\ldots,k-1$. Let $\theta\in[-1,\alpha_0]$. If $\alpha_k-kp+1=0$, then, for each $u\in X^{k,p}_\infty$,
\begin{equation*}
\int_0^\infty\exp_p\left(\mu|u|^{\frac{p}{p-1}}\right)r^\theta\mathrm dr<\infty,\quad\forall \mu\geq0.
\end{equation*}
\end{lemma}
\begin{proof}
Using Fubini's theorem, we obtain that
\begin{align}
\int_0^\infty\exp_p\left(\mu|u|^{\frac{p}{p-1}}\right)r^\theta\mathrm dr&=\sum_{j=0}^\infty\dfrac{\mu^{p-1+j}}{\Gamma(p-1+j)}\int_0^\infty|u|^{\frac{p}{p-1}(p-1+j)}r^\theta\mathrm dr\nonumber\\
&\leq\sum_{j=0}^\infty\dfrac{\left(\mu C_j^{\frac{p}{p-1}}\|u\|_{X^{k,p}_\infty}^{\frac{p}{p-1}}\right)^{p-1+j}}{\Gamma(p-1+j)}\label{eqasfkjna},
\end{align}
where $C_j$ is the constant given by Theorem \ref{theoimersaoinfinito} such that
\begin{equation*}
\left(\int_0^\infty|u|^{\frac{p}{p-1}(p-1+j)}r^\theta\mathrm dr\right)^{\frac{p-1}{p(p-1+j)}}\leq C_j\|u\|_{X^{k,p}_\infty},\quad\forall j\in\mathbb N\cup\{0\}.
\end{equation*}
Finally, we conclude our lemma by applying the d'Alembert's ratio test on \eqref{eqasfkjna} together with the expression of $C_j$.
\end{proof}

Our approach involves using a minimization argument to solve this problem. Therefore, we need to find a minimizer under a constraint for the associate functional $\Phi\colon X^{2,2}_\infty\to\mathbb R$, which is given by
\begin{equation*}
\Phi(u)=\dfrac12\int_0^\infty\left|\Delta_{\theta}u\right|^2r^{\theta}\mathrm dr-\int_0^\infty F(r,u)r^\eta\mathrm dr,
\end{equation*}
where $F(r,s)=\int_0^sf(r,t)\mathrm dt$. However, before proving the existence of a weak solution, we need the subsequent compactness theorem.

\begin{theo}\label{theocomp}
Suppose the assumptions of Theorem \ref{theoexactk}, $\alpha_0>-1$, and $F\colon[0,\infty)\times[0,\infty)\to[0,\infty)$ is a function such that $F(X)$ is bounded for $X\subset[0,\infty)\times[0,\infty)$ bounded. Furthermore, assume that
\begin{equation}\label{h1}
\lim_{t\to0}F(r,t)t^{-p}=0,\mbox{ uniformly on }r,
\end{equation}
and
\begin{equation}\label{h2}
\lim_{t\to\infty}F(r,t)t^{\frac{p}{p-1}}\exp\left(-\beta_{0,k}t^{\frac{p}{p-1}}\right)=0,\mbox{ uniformly on }r,
\end{equation}
where $\beta_{0,k}$ is given in Theorem \ref{theoexactk}. For any sequence $(u_n)$ weakly converging to some $u_0$ in $X^{k,p}_\infty$ satisfying $\|\nabla^k_Lu_n\|_{L^p_{\alpha_k+\lfloor\frac{k}2\rfloor(\theta-\gamma)p}}\leq1$, we have
\begin{equation*}
\int_0^\infty F(r,|u_n|)r^\eta\mathrm dr\overset{n\to\infty}\longrightarrow\int_0^\infty F(r,|u|)r^\eta\mathrm dr.
\end{equation*}
\end{theo}
\begin{proof}
From a radial lemma \cite[Lemma 2.2]{arXiv:2306.00194} and $\alpha_0>-1$, we obtain $u_n(r)\to0$ as $r\to\infty$ uniformly on $n$. This, together with \eqref{h1}, imply that for any $\varepsilon>0$, there exists $R>0$ such that
\begin{equation}\label{lu33}
\int_R^\infty F(r,|u_n|)r^\eta\mathrm dr\leq \varepsilon\int_0^\infty|u_n|^pr^\eta\mathrm dr\leq C\varepsilon\mbox{ and }\int_R^\infty F(r,|u_0|)r^\eta\mathrm dr\leq\varepsilon.
\end{equation}
On the other hand, for any $\varepsilon>0$, using \eqref{h2} and \cite[Lemma 3.1]{arXiv:2306.00194}, we obtain $A>0$ such that
\begin{equation*}
F(r,t)\leq \varepsilon\dfrac{\exp_p\left(\beta_{0,k}t^{\frac{p}{p-1}}\right)}{(1+t)^{\frac{p}{p-1}}},\quad\forall t>A,r\in[0,\infty).
\end{equation*}
Then, by Theorem \ref{theoexactk},
\begin{equation}\label{lu34}
\int_{|u_n|>A}F(r,|u_n|)r^\eta\mathrm dr\leq C\varepsilon\int_0^\infty|u_n|^pr^\eta\mathrm dr\leq C\varepsilon\mbox{ and }\int_{|u_n|>A}F(r,|u_0|)r^\eta\mathrm dr\leq C\varepsilon.
\end{equation}
Combining \eqref{lu33} with \eqref{lu34}, we get
\begin{equation*}
\left|\int_0^\infty F(r,|u_n|)r^\eta\mathrm dr-\int_0^\infty F(r,|u_0|)r^\eta\mathrm dr\right|\leq C\varepsilon+\left|\int_{|u_n|\leq A,r\leq R}\left[F(r,|u_n|)-F(r,|u_0|)\right]r^\eta\mathrm dr\right|,
\end{equation*}
for all $n\in\mathbb N$. Finally, applying the Lebesgue Dominated Convergence Theorem, we conclude
\begin{equation*}
\int_0^\infty F(r,|u_n|)r^\eta\mathrm dr\overset{n\to\infty}\longrightarrow\int_0^\infty F(r,|u_0|)r^\eta\mathrm dr
\end{equation*}
\end{proof}

We have obtained the weak solution for \eqref{eqprobleml}, which serves as a specific instance of the case \eqref{eqproblem}. This outcome is presented in Proposition \ref{prop81}. Additionally, we establish the regularity theory for the general problem \eqref{eqproblem}, achieving regularity throughout the interval $(0,\infty)$ as demonstrated by Proposition \ref{propcs}, and completing the regularity at the origin as described by Proposition \ref{propclsol1}.

\begin{prop}\label{prop81}
Assume that $t\mapsto f(r,t)$ is an odd function with $f(r,t)\geq0$, and $F$ satisfies the conditions of Theorem \ref{theocomp}. Set
\begin{equation*}
m_\infty:=\sup_{\underset{\|\Delta_\theta u\|_{L^2_{\theta}}=1}{u\in X^{2,2}_\infty}}\int_0^\infty F(r,u)r^\eta\mathrm dr.
\end{equation*}
Then, $m_\infty$ is attained by a function $u_0\in X^{2,2}_\infty$ with $\|\Delta_{\theta}u_0\|_{L^2_{\theta}}=1$. Moreover, $u_0$ is a weak solution of
\begin{equation*}
\left\{\begin{array}{l}
\Delta^2_{\theta}u=\lambda^{-1}f(r,u)r^{\eta-\theta}\mbox{ in }(0,\infty),\\
u(0)=u'(0)=0,
    \end{array}\right.
\end{equation*}
where $\lambda=\int_0^Rf(r,u_0)u_0r^\eta\mathrm dr$.
\end{prop}
\begin{proof}
Let $(u_n)$ be a maximizing sequence for $m_\infty$ with $\|\Delta_\theta u_n\|_{L^2_{\theta}}=1$. From Theorem \ref{propequivnormkgrad}, we have that $(u_n)$ is bounded on $X^{2,2}_\infty$. Using Theorem \ref{theocomp}, we obtain $u_0\in X^{2,2}_\infty$ with $\|\Delta_\theta u_0\|_{L^2_{\theta}}\leq1$ such that
\begin{equation}\label{eq816}
m_\infty=\int_0^\infty F(r,|u_0|)r^\eta\mathrm dr.
\end{equation}

Since $f(r,t)=-f(r,-t)\geq0$ for all $t\geq0$, we obtain $F(r,u)\geq0$ and $F(r,\tau u_0)\geq F(r,u_0)$ for all $\tau\geq1$. We can assume that $F(r,u)$ is not always zero. Then, $u_0\in X_{\infty}^{2,2}\backslash\{0\}$ because $u_0$ is a maximum of $m_\infty$. We can assume $\|\Delta_\theta u_0\|_{L^2_{\theta}}=1$. Otherwise, defining $\widetilde u_0=u_0/\|\Delta_\theta u_0\|_{L^2_{\theta}}$, we have $\|\Delta_\theta \widetilde u_0\|_{L^2_{\theta}}=1$ and
\begin{equation*}
\int_0^\infty F(r,\widetilde u_0)r^\eta\mathrm dr\geq\int_0^\infty F(r,u_0)r^\eta\mathrm dr=\sup_{\underset{\|\Delta_\theta u\|_{L^2_{\theta}}=1}{u\in X^{2,2}_\infty}}\int_0^\infty F(r,u)r^\eta\mathrm dr.
\end{equation*}
Therefore, we obtain $u_0\in X^{2,2}_{\infty}$ that attains the supremum with $\|\Delta_\theta u_0\|_{L^2_{\theta}}=1$. Moreover, applying Lagrange Multipliers in \eqref{eq816}, we conclude
\begin{equation*}
\int_0^\infty \Delta_\theta u_0\Delta_\theta vr^{\theta}\mathrm dr=\lambda^{-1}\int_0^Rf(r,u_0)vr^\eta\mathrm dr\quad\forall v\in X_{\infty}^{2,2},
\end{equation*}
where $\lambda=\int_0^Rf(r,u_0)u_0r^\eta\mathrm dr$.
\end{proof}

\begin{prop}\label{propcs}
Suppose $u_0$ is a weak solution of \eqref{eqproblem}. Then $u_0\in C^{4}(0,\infty)$, $\Delta_\theta u_0\in C^{2}(0,\infty)$, and $\Delta^2_{\theta}u_0=f(r,u_0)r^{\eta-\theta}$ $\forall r>0$. Moreover, $u_0(r)\overset{r\to\infty}\longrightarrow0$ if $\alpha_0\geq\alpha_1-2\geq-1$, and $\Delta_\theta u_0(r)\overset{r\to\infty}\longrightarrow0$ if $\alpha_0\in(-1,2\theta-3)$ and $\alpha_1\in[1,2\theta-1)$.
\end{prop}
\begin{proof}
We claim that $\Delta_\theta u_0$ has a weak derivative given by
\begin{equation}\label{wdlu}
(\Delta_\theta u_0)'(r)=-r^{-\theta}\int_0^rf(s,u_0(s))s^\eta\mathrm ds.
\end{equation}
Note that the right term is well-defined by Lemma \ref{proptobeproved} and Theorem \ref{theo1}. Using \eqref{wdlu} and that the right term of \eqref{wdlu} is $C^1(0,\infty)$, we conclude our proposition except by 
$u_0(r)\overset{r\to\infty}\longrightarrow0$ and $\Delta_\theta u_0(r)\overset{r\to\infty}\longrightarrow0$. Firstly, let us prove \eqref{wdlu}.

Let $\varphi\in C^\infty_0(0,\infty)$. Writing $v(r)=-\int_r^\infty\varphi s^{-\theta}\mathrm ds$, we have $\varphi=r^\theta v'$ and $v\in X_\infty^{2,2}$. Using that $u_0$ is a weak solution and Fubini's Theorem, we get
\begin{align*}
\int_0^\infty \Delta_\theta u_0\varphi'\mathrm dr&=-\int_0^\infty \Delta_\theta u_0\Delta_\theta vr^\theta\mathrm dr\\
&=\int_0^\infty\int_r^\infty f(r,u_0(r))\varphi(s)s^{-\theta}r^{\eta}\mathrm ds\mathrm dr\\
&=\int_0^\infty s^{-\theta}\int_0^sf(r,u_0(r))r^\eta\mathrm dr\varphi(s)\mathrm ds.
\end{align*}
This concludes \eqref{wdlu}.

The fact that $u_0(r)\overset{r\to\infty}\longrightarrow0$ follows by applying \cite[Lemma 2.2]{arXiv:2306.00194} with $\alpha_0\geq\alpha_1-2\geq-1$. Finally, let us check that $\Delta_\theta u_0(r)\overset{r\to\infty}\longrightarrow 0$. Fix $v\colon[0,\infty)\to\mathbb R$ smooth such that $v\equiv0$ in $[0,1]$ and $v(r)=r^{1-\theta}/(\theta-1)$ in $[2,\infty)$. Note that $v\in X^{2,2}_\infty$ by $-1<\alpha_0<2\theta-3$ and $-1\leq\alpha_1<2\theta-1$. Using that $u_0$ is a weak solution of \eqref{eqproblem}, we have
\begin{align*}
\int_0^\infty f(r,u_0)vr^\eta\mathrm dr&=\int_0^\infty \Delta_\theta u_0\Delta_\theta vr^\theta\mathrm dr=-\int_0^\infty \Delta_\theta u_0(r^\theta v')'\mathrm dr\\
&=\lim_{R\to\infty}\Delta_\theta u_0(R)+\int_0^\infty r^\theta(\Delta_\theta u_0)'v'\mathrm dr\\
&=\lim_{R\to\infty}\Delta_\theta u_0(R)-\int_0^\infty\left(r^\theta(\Delta_\theta u_0)'\right)'v\mathrm dr\\
&=\lim_{R\to\infty}\Delta_\theta u_0(R)+\int_0^\infty \Delta_\theta ^2u_0vr^\theta\mathrm dr.
\end{align*}
Since $\Delta_\theta^2u_0=r^{\eta-\theta}f(r,u_0)$, we conclude that $\lim_{R\to\infty}\Delta_\theta u_0(R)=0$ and the proof of the proposition.
\end{proof}

\begin{prop}\label{propclsol1}
Suppose $u_0$ is a weak solution of \eqref{eqproblem} with $\eta>1$. Then $u_0\in C^4(0,\infty)\cap C^3([0,\infty))$ is a classical solution of \eqref{eqproblem}. Moreover, $\Delta_\theta u_0\in C^2(0,\infty)\cap C^1([0,\infty))$, and the following initial conditions hold: $u'_0(0)=(\Delta_\theta u_0)'(0)=0$. Additionally, we find that $u''_0(0)=-\Delta_\theta u_0(0)/(\alpha+1)$ and $u'''_0(0)=0$.
\end{prop}
\begin{proof}
We claim that
\begin{equation}\label{cl1}
\Delta_\theta^2u_0=r^{\eta-\theta}f(r,u_0)\in L^p_{3p-1-\theta p},\quad\forall p>1.
\end{equation}
Indeed, through the application of $[\exp_2(t)]^p\leq\exp_2(pt)$, we deduce
\begin{equation*}
\int_0^\infty|r^{\eta-\theta}f(r,u_0)|^pr^{3p-1-\theta p}\mathrm dr\leq C\int_0^\infty \exp_2\left(\beta p u_0^2\right)r^{-1}\mathrm dr.
\end{equation*}
Hence, Lemma \ref{proptobeproved} guarantees the validity of \eqref{cl1}.

Using Lemma \ref{lemmafjs} in \eqref{cl1}, we obtain $\Delta_{\theta}u_0\in X^{2,p}_\infty(p-1-\theta p,2p-1-\theta p,3p-1-\theta p)$. Similarly, from Lemma \ref{lemmafjs}, we deduce that $u_0\in X^{2,p}_\infty(-p-1-\theta p,-1-\theta p,p-1-\theta p)$. As a result, the Morrey case of Theorem \ref{theo32} guarantees $u_0\in C^1([0,\infty))$. Proposition \ref{propcs} implies that $u_0\in C^4(0,\infty)\cap C^1([0,\infty))$ with $\Delta_\theta u_0\in C^2(0,\infty)$ and $\Delta_\theta^2u_0=r^{\eta-\theta}f(r,u_0)$ for all $r\in(0,\infty)$. We proceed to verify the initial condition $u_0'(0)=(\Delta_\theta u_0)'(0)=0$. Applying L'Hopital's rule in \eqref{wdlu} with $\eta>1$, we have
\begin{equation}\label{ganojdfna}
\lim_{r\to0}(\Delta_\theta u_0)'(r)=-\lim_{r\to0}\dfrac{f(r,u_0(r))r^\eta}{\theta r^{\theta-1}}=0.
\end{equation}
Using $u_0'(r)=-r^{-\theta}\int_0^r\Delta_\theta u_0(s)s^\theta\mathrm ds$ and L'Hopital's rule twice, we obtain
\begin{equation*}
\lim_{r\to0}u'_0(r)=-\theta^{-1}\lim_{r\to0}\Delta_\theta u_0(r)r=\theta^{-1}\lim_{r\to0}(\Delta_\theta u_0)'(r)r^2=0.
\end{equation*}
By \eqref{wdlu} and \eqref{ganojdfna}, we conclude that $\Delta_\theta u_0\in C^1([0,\infty))$.

Now we are left to show that $u_0\in C^3([0,\infty))$ with $u_0''(0)=-\Delta_\theta u_0(0)/(\theta+1)$ and $u_0'''(0)=0$. We notice that
\begin{equation*}
u_0''(r)=-\theta\dfrac{u_0'(r)}r-\Delta_\theta u_0(r)=\theta r^{-\theta-1}\int_0^r\Delta_\theta u_0(s)s^\theta\mathrm ds-\Delta_\theta u_0(r)
\end{equation*}
and hence $\lim_{r\to0}u''_0(r)=-\Delta_\theta u_0(0)/(\theta+1)$. Furthermore,
\begin{align*}
u_0'''(r)&=\theta r^{-1}\Delta_\theta u_0(r)-\theta(\theta+1)r^{-\theta-2}\int_0^r\Delta_\theta u_0(s)s^\theta\mathrm ds-(\Delta_\theta u_0)'(r)\\
&=\dfrac{\theta(\theta+1)}{r^{\theta+2}}\int_0^r\left(\Delta_\theta u_0(r)-\Delta_\theta u_0(s)\right)s^\theta\mathrm ds-(\Delta_\theta u_0)'(r).
\end{align*}
Given $\varepsilon>0$, $(\Delta_\theta u_0)'(0)=0$ implies that there exists $\delta>0$ such that $|(\Delta_\theta u_0)'(r)|\leq\varepsilon$ and $|\Delta_\theta u_0(r)-\Delta_\theta u_0(s)|\leq\varepsilon(r-s)$ for all $0<s\leq r<\delta$. Thus,
\begin{equation*}
|u_0'''(r)|\leq\varepsilon\dfrac{\theta(\theta+1)}{r^{\theta+1}}\int_0^rs^\theta\mathrm ds+\varepsilon=(\theta+1)\varepsilon,\quad\forall r\in(0,\delta).
\end{equation*}
Consequently, $u_0\in C^3([0,\infty))$ with $u_0'''(0)=0$. Therefore, we completed the proof.
\end{proof}

\begin{proof}[Proof of Theorem \ref{theoapp}]
The theorem's proof directly follows from Propositions \ref{prop81}, \ref{propcs}, and \ref{propclsol1}.
\end{proof}

\end{document}